\documentclass[12pt,a4paper,fleqn]{article}

\usepackage{rotating}
\usepackage[ruled]{algorithm}
\usepackage{color}
\usepackage{cite}
\usepackage{amssymb}
\usepackage{amsmath}
\usepackage{bm}
\usepackage{xcolor}
\usepackage{mdframed}
\usepackage{amsmath,amsthm}
\usepackage{amsfonts}
\usepackage{graphicx}
\usepackage{graphics}
\usepackage{setspace}
\usepackage{enumitem}
\usepackage[left=1cm, right=1cm, top=2.50cm, bottom=4.30cm]{geometry}
\usepackage{roundbox}

\usepackage[pagebackref=false,colorlinks,linkcolor=blue,citecolor=magenta]{hyperref}
\usepackage{listings}

\usepackage{titlesec}

\newcommand{\E}{\mathbb{E}}
\newcommand{\lamn}{\lambda^n}

\usepackage{algorithm}
\usepackage{algpseudocode}
\numberwithin{equation}{section}
\newtheorem{Theorem}{Theorem}[section]

\newtheorem{Proposition}[Theorem]{Proposition}

\newtheorem{lemma}[Theorem]{Lemma}

\newenvironment{Proof}{\noindent{\bf Proof.}}{\hfill{$\blacksquare$}}

\begin{document}
\date{\small\textsl{\today}}
\title{Hedging of exotic options in Hawkes jump-diffusion models by Malliavin calculus}
\author{\large
\large  Ayub Ahmadi, Mahdieh Tahmasebi\footnote{Corresponding author. {\em E-mail
addresses:} tahmasebi@modares.ac.ir (M.Tahmasebi), Ayubahmadi@modares.ac.ir (A.Ahmadi).} \\
\small{\em $^{\mbox{}}$\em Department of Applied Mathematics,
Faculty of Mathematics,}\\
\small{\em Tarbiat Modares University, Tehran, Iran}\vspace{-0.5mm}\\
}\maketitle
\vspace{-2mm}
\linethickness{.5mm}{\line(1,0){500}}
\begin{abstract}
In financial mathematics, the calculatinon of the Greeks, and in particular the delta, is emphasized, due to its role in risk managment. In this article, we employ Malliavin calculus to determine the delta of  European and Asian options, where the underlying asset evolves according to a Hawkes jump-diffusion process. A central feature is that the Hawkes jump intensity is stochastic, which substantially affects the delta representation.\\
\textbf{Keys: Malliavin calculus, Stochastic intensity, Hawkes processes, Delta heding. }
\vspace{.5cm}\\
\end{abstract}
\linethickness{.5mm}{\line(1,0){500}}
\section{Introduction}
Stochastic jump models with stochastic intensity describe the behavior of phenomena that can experience abrupt changes in value. They rely on stochastic processes to describe the evolution of a variable over time in a probabilistic framework. A central feature is the ability to accommodate sudden, unpredictable shifts, unlike traditional stochastic models, such as Brownian motion, which assume smooth and continuous over time. In these models, jumps occur at random times and with random sizes, usually modeled using a Poisson distribution. The stochastic intensity governs the jump rate and itself evolves as a stochastic process over time. These models are widely used in finance to capture both the gradual evolution of asset prices and the abrupt jumps caused by unexpected events (
\cite{huang2014option, ditstochastic, brigo2005credit, brigo2006credit, feng2017cva, bianchi2015investigating, leung2009counterparty}).
They have applications in a wide range of fields, including finance, insurance, and engineering (see, for example,
\cite{n1,n2}).
A key application of stochastic intensity jump models is in the credit risk modeling for financial instruments. In this context, the jumps represent defaults or other credit events, while the stochastic intensity expresses the evolving rate  at which such events occur.
\\
Here, we use a special class of stochastic intensity model known as a Hawkes process. Hawkes introduced the Hawkes process 
\cite{n24, 1971b}, 
as a model for event data arising from contagious phenomena. 
Earlier work has explored applications in engineering and reliability theory (see, e.g., \cite{Rangan and Grace 1988} and references therein). Today, Hawkes processes are applied across diverse fields: geology (earthquakes, volcanic eruptions), biology (population growth, the spread of infections, neuronal activity), computer science and social sciences (networks and social interactions), and finance (order-book dynamics, defaults and related topics) (see, e.g., \cite{Zhuang 2002, Reynaud-Bouret et al. 2013, Delattre 2016, Hawkes 2018, n21}).\\
The main feature of this class of processes is the self-exciting behavior whereby the jump intensity increases immediately after each event. It is reasonable to limit such a self-exciting effect to decay over time. This behavior is captured by the exponential Hawkes model, which provides a tractable representation.
$$
{\lambda}_{t}=\lambda_{0}+\alpha \sum_{T_{i}<t} e^{-\beta\left(t-T_{i}\right)},
$$
where $t \in \mathbf{R}_{+}$ denotes time, $T_{i}$ the random time of the $i$-th jump and $\lambda_{0}, \alpha, \beta$ are deterministic, constant and positive parameters. Due to its Markov property, the exponential Hawkes model offers mathematical tractability  (see \cite{foschi2021measuring}). \\
By definition, a Hawkes process with $\alpha=0$ reduces to a Poisson process. It is therefore useful to study the behavior of the Hawkes process, particularly in terms of its deviation from a Poisson process when $\alpha \approx 0$.
Conversely, a Hawkes process approaches a Poisson process, as $\beta \longrightarrow +\infty$. However, no results quantify their similarity for finite $\beta$.
In the exponential Hawkes model, $\alpha$ represents the immediate impact of a jump on the intensity, and $\beta$ governs its decay over time. \\
Hawkes jump–diffusion models produce heavier-tailed distributions with higher peaks than Poisson jump-diffusion models, which leads to higher option prices under the Hawkes framework for deep out-of-the-money options. Asset price dynamics exhibiting jumps with self-exciting features have been studied, for example, by \cite{Filimonov et al. (2014)}, where sequences of large price moves is triggered by prior large moves. Fulop et al. \cite{Fulop et al. (2015)} propose a Bayesian learning approach to jump-cluster detection and provide evidence of jump clustering since the 1987 market crash, with even more pronounced clustering after the 2008 global financial crisis. Rambaldi et al. \cite{Rambaldi et al. (2015)} describe foreign exchange markets, while Kiesel and Paraschiv \cite{1} describe power price dynamics in energy markets using Hawkes-type models. Bernis et al. \cite{2} perform a sensitivity analysis to the Hawkes parameters within a credit-risk framework. Liu and Zhu in
\cite{liu2019pricing}
present an analytical method for pricing variance swaps when the underlying asset follows a Hawkes jump-diffusion process with stochastic volatility and clustered jumps.
In \cite{ma2020pricing}, the authors apply a Hawkes jump–diffusion model to evaluate foreign equity options and derive valuation formulas using the Fourier transform, accommodating both clustered jumps and cross-market jump propagation. They calculate the Greeks and hedging strategies for swap options within this framework.
Brignone and Sgarra \cite{brignone2020asian}
introduce a method for pricing Asian options in risky-asset models with Hawkes processes and an exponential kernel.
Ketelbuters and Hainaut \cite{ketelbuters2022cds} investigate CDS pricing in fractional Hawkes models. Recently, Chen et. al. \cite{chen2024modeling} modeled the clustering behavior of price jumps and variance using high-frequency data within a labeled Hawkes process in a bivariate jump-diffusion setting. \\
Malliavin calculus plays a critical role in financial mathematics, especially in the pricing and risk management of financial derivatives. It is a powerful tool for computing the Greeks, which are measures of sensitivity of the derivative prices to changes in the underlying asset price, volatility, time to maturity, and interest rates, with greater precision.  This calculus provides a rigorous framework for considering the stochastic behavior of the underlying assets and deriving derivative prices. It has been widely applied in the literature (see, for example, \cite{n3, n4, n5, n6, n7, n8, n10, n11, n12, n13, n14, n15, n16, n17}). In Hawkes models,  Torrisi \cite{n20} provides a Gaussian bound for the first chaos of these point processes by combining Stein's method and Malliavin calculus via a Clark–Ocone type representation formula.\\
To the best of our knowledge, this is the first paper to propose a hedging strategy for exotic options in Hawkes models.  Our contributions in this manuscript are as follows. We replace the stochastic integral with respect to a point process driven by a stochastic intensity with a stochastic integral with respect to a Lévy process whose Lévy measure incorporates the stochastic intensity (see \eqref{3} and \eqref{7}). This reformulation enables us to express the solution of Hawkes SDEs in an exponential form of another process, which is convenient for obtaining an explicit expression for their Malliavin derivatives. Despite classical arguments, the Malliavin derivatives of the stochastic intensity appear in the formula (Equation \eqref{DNX}), which complicates the proofs of moment boundedness for the inverse of the solution (see Lemma \ref{varoon}). Invertibility of the Malliavin derivatives is required to show that the desired Malliavin weights, used to obtain deltas for Asian options and European options, lie in the domain of the Skorokhod integral (see Lemmas \ref{deltadomain} and \ref{deltadomainY}). In addition, numerically, computing the delta needs establishing the convergence of the discretized method. In our numerical approach, for every discretized path, we update the point at the next time step using the stochastic intensity evaluated at the current time (see Section 5).\\
This article is organized as follows. Section 2 recalls the definition of Hawkes processes and examines the properties of their stochastic intensity. Section 3 introduces the main SDEs and establishes properties of their solutions. In line with the structure of Malliavin derivatives, we adopt a novel technique to prove the boundedness of the inverse of the derivatives, a step necessary to show that the Malliavin weight appearing in the delta computation lies within the domain of the Skorokhod integral. Section 4 presents our main results on delta computation via innovation quadratures for both European and Asian options. The convergence of our numerical methods used to compute the delta is proved in Section 5. Section 6 provides numerical illustrations and verification of our results.
Most proofs of lemmas and the preliminaries on the Malliavin calculus are deferred to the Appendices.

\section{Hawkes processes and properties of stochastic intensities }
In this section, we give some preliminaries on point processes with stochastic intensity, as mentioned by B{\'e}rmaud in Chapter 5 of \cite{bremaud}. Especially, Hawkes processes are generally a class of multivariate point processes introduced by Hawkes in \cite{1971b} and \cite{n24} that are characterized by a stochastic intensity vector. Some of their applications can be seen in \cite{n22}. We refer the reader to \cite{n20} for more details.
\\
The process $N=\{N(A)\}_{A\in\mathbb{\mathcal{B}(R)}}:=\{\sum_{n\in\mathbb{Z}} 1_{A}(T_{n}) \}_{A\in\mathbb{\mathcal{B}(R)}}$ called the point process with increasing random times $\{T_{n}\}_{n\in\mathbb{Z}}$. We suppose that $N$ is locally finite and simple, and $\mathcal{F}_{t}:=\{\mathcal{F}_{t}\}_{t\in\mathbb{R}}\subset\mathcal{F}$ be a filtration such that $\mathcal{F}_{t}\supseteq{\mathcal{F}^{_{N}}_{t}}$ for any $t\in\mathbb{R},$ where $\mathcal{F}^{N}:=\{\mathcal{F}^{N}_{t}\}_{t\in\mathbb{R}}$ is the natural filtration of the point process $N$, that is,
\begin{equation*}
\nonumber\mathcal{F}^{N}_{t}:=\sigma\{N(A):A\in\mathbb{\mathcal{B}(R)},A\subseteq(-\infty,t]\},
\end{equation*}
Given a nonnegative $\mathcal{F}_{t}$-adapted stochastic process $\{\lambda(t)\}_{t\in\mathbb{R}}$
such that 
\begin{equation}\label{finite}
\int_{a}^{b}\lambda(t)\;dt<\;\infty, \;\; \text{a.s.} \;\; \text{for all}\;\;\; a,b\in\mathbb{R}, 
\end{equation}
we call $\{\lambda(t)\}_{t\in\mathbb{R}}$ as a $\mathcal{F}$-stochastic intensity of $N$, if for any $a, b \in\mathbb{R},$
\begin{equation*}
\nonumber\mathbb{E}\Big[N((a,b]) |\mathcal{F}_{a}\Big]=\mathbb{E}\Big[{\int_{a}^{b}}\lambda(t)dt|\mathcal{F}_{a}\Big],\;\;\;\;\;\text{a.s.}
\end{equation*}
In this article, given a Poisson process $\overline{N}$ in
$\mathbb{R}\times\left[0,\infty\right)$ with mean measure $dtdz$, we consider the point process 
\begin{equation*}
N(dt):=\overline{N}(dt\times\left(0,\lambda(t)\right]),
\end{equation*}
where $\{\lambda(t)\}_{t\in\mathbb{R}}$ is a nonnegative process of the form
\begin{align}\label{3}
	\nonumber
\lambda(t)& :=\lambda_{0}+\int_{0}^{t}\alpha{e^{-\beta(t-s)}}dN_s=\lambda_{0}+\sum_{t_i<t}\alpha{e^{-\beta(t-t_i)}}\\
&=\varphi(t,\overline{N}\lvert_{\left(-\infty,t\right)})=\varphi_t(\overline{N}\lvert_{\left(-\infty,t\right)}) =  \lambda_{0}+\int_{0}^{t}\int_{\mathbb{R}_+}\alpha{e^{-\beta(t-s)}}1_{(0, \lambda(s)]}(z)\overline{N}(dz, ds),
\end{align}
in which $\varphi:\mathbb{R}\times\mathcal{N}\to\mathbb{R_{+}}$ is a measurable functional such that for all $ a,b\in\mathbb{R}$, $\lambda(t)$ satisfying \eqref{finite}.
Here, $\mathcal{N}$ denotes the space of simple and locally finite counting measures on $\mathbb{R}\times\mathbb{R_{+}}$ endowed with the vague topology.
For simplicity, with a little abuse of notation, we denote by $\overline{N}\lvert_{\left(-\infty,t\right)}$ the restriction of $\overline{\mathit{N}}$ to ${\left(-\infty,t\right)}\times\mathbb{R_{+}}$, that is,
\begin{equation}
\nonumber \overline{N}\lvert_{(\left(-\infty,t\right))}(A):=\overline{N}(A\cap({\left(-\infty,t\right)}\times\mathbb{R_+})),\;\; A\in\mathcal{B}(\mathbb{R})\otimes\mathcal{B}(\mathbb{R_+}).
\end{equation}
According to Lemma 2.1 in \cite{n20}, $N$ has $\mathcal{F^{\overline{N}}}$-stochastic intensity $\{|\lambda(t)|\}$, where  denote $\mathcal{F^{\overline{N}}}=\{\mathcal{F}_t^{\overline{N}}\}_{t\in\mathbb{R}}$
the natural filtration of $\overline{N}$, i.e.
\begin{equation}
\nonumber \mathcal{F}_t^{\overline{N}}=\sigma\{\overline{N}(A\times B):A\in\mathcal{B}(\mathbb{R}),B\in\mathcal{B}(\left[0,\infty\right)),A\subseteq(-\infty,t]\}.
\end{equation}
For every $s,t \in [0,T]$, due to the proof of Theorem 3.1 in \cite{n20}, 
\begin{align*} 
\mathbb{E}\Big(\int_t^s{N}(du)\vert \mathcal{F}^{\overline{N}}_t\Big)&=\mathbb{E}\Big(\int_t^s\int_{{\mathbb{R}_+}}\overline{N}(du\times\left(0,\lambda(u)\right])\vert \mathcal{F}^{\overline{N}}_t\Big)=\mathbb{E}\Big(\int_t^s\int_{{\mathbb{R}_+}}1_{\left(0,\lambda(u)\right]}(z)\overline{N}(du\times dz)\vert \mathcal{F}^{\bar{N}}_t\Big)\\
&=\mathbb{E}\Big(\int_t^s \int_{\mathbb{R}_+} 1_{\left(0,\lambda(u)\right]}(z)dzdu\vert \mathcal{F}^{\overline{N}}_t\Big)=
\mathbb{E}\Big(\int_t^s\lambda(u) du\vert \mathcal{F}^{\overline{N}}_t\Big).
\end{align*}
In addition, for every $0 \leq t \leq s \leq T$ and for every $\mathcal{F}^{\overline{N}}_t$-predictable function $k$  
\begin{equation}\label{kN}
\mathbb{E}\Big(\int_t^s \int_{\mathbb{R}_+}k(u,z) N(du)\vert  \mathcal{F}^{\overline{N}}_t \Big)=\mathbb{E}\Big(\int_t^s \int_{\mathbb{R}_+}k(u,z) 1_{\left(0,\lambda(u)\right]}(z)dzdu\vert \mathcal{F}^{\overline{N}}_t\Big),
\end{equation}
and if $\mathbb{E}\Big(\int_0^t \int_{\mathbb{R}_+} (k(s,z))^2 1_{\left(0,\lambda(s)\right]}(z)dz ds\Big)< \infty$, for given $\mathcal{G}$-measurable intensity $\lambda$ 
\begin{equation}\label{5}
\mathbb{E}\Big(exp\{iu\int_0^t \int_{\mathbb{R}_+}k(s,z) N(ds) \}\Big \vert \mathcal{G}\Big)=exp\Big\{\int_0^t\int_{\mathbb{R}_+}
(e^{iuk(s,z)}-1)1_{(0, \lambda(s)]}(z)dzds\Big\}.
\end{equation}
We refer the reader to Chapter 5 of \cite{bremaud} for more details. Let us first we derive some properties of the stochastic intensity mentioned in \eqref{3}. The process $\lambda(t)$ can be also demonstrated as the solution of the SDE 
\begin{equation*}
d\lambda({t})=\beta(\lambda_{0}-\lambda({t}))dt+\alpha dN_t.
\end{equation*}
Note that the stability condition $\alpha\int_{0}^{\infty}e^{-\beta s}ds=\frac{\alpha}{\beta}<1$ should be held, (see \cite{banos2022change}, Section 2).
\begin{lemma}\label{ppp0}
For every $p \geq 2$,
\begin{equation*}
\mathbb{E}\Big(\sup_{0 \leq t \leq T}e^{p\beta t}\vert \lambda(t) \vert^p \Big)<\infty, \qquad and  \qquad \mathbb{E}\Big(\sup_{0 \leq t \leq T}e^{p\lambda(t)}\Big)<\infty.
\end{equation*}
\end{lemma}
\begin{proof}
First we can deduce from the definition of $\lambda(t)$ and its stability, $\alpha < \beta$ that for every $0 \le t \le T$, 
\begin{equation}\label{lamb1}
\mathbb{E}\Big(\int_0^t \lambda(s) ds\Big) = \frac{\lambda_0}{\alpha-\beta}\Big(\alpha e^{(\alpha-\beta)t} -\beta\Big)< \lambda_0\frac{\beta}{\beta- \alpha}.
\end{equation}
From \eqref{kN}, Burkholder–Davis–Gundy inequality and \eqref{lamb1}, there exists some constant $C_p$ such that 
\begin{align*}
\mathbb{E}\Big(\sup_{0 \leq s \leq t}e^{p\beta s}\lvert{\lambda(s)}\rvert^{p}\Big)  &\le  (1+\beta^p (e^{p\beta t}-1) )\lambda_0^{p}+C_p\mathbb{E}\Big(\int_{0}^{t}\int_{{\mathbb{R}_+}} e^{2\beta s}1_{\left(0,\lambda(s)\right]}^{2}(z)dzds\Big)^{\frac{p}{2}}\\
		&+ C_p\mathbb{E}\Big(\int_{0}^{t}\int_{{\mathbb{R}_+}} e^{p\beta s}1_{\left(0,\lambda(s)\right]}^{p}(z)dzds\Big)\\
		& \le  (1+\beta^p (e^{p\beta t}-1) )\lambda_0^{p}+ C_pt^{\frac{p}{2}-1}\mathbb{E}\Big(\int_{0}^{t}e^{p\beta s}\lambda(s)^{\frac{p}{2}} ds\Big)+ C_pe^{p\beta T}\mathbb{E}\Big(\int_{0}^{t}\lambda(s)ds\Big)\\
		&\leq A_1+ \frac12 C_pt^{\frac{p}{2}-1}\int_{0}^{t}\mathbb{E}\Big(\sup_{0 \leq s \leq t}e^{p\beta s} \lambda(s)^{p} ds\Big)+ \frac12 e^{p\beta t} C_pt^{\frac{p}{2}}.
\end{align*}	
where $A_1:= (1+\beta^p (e^{p\beta t}-1) )\lambda_0^{p}+e^{p\beta T}C_p\lambda_0\frac{\beta}{\beta- \alpha}$. The Gronwall inequality results the first part of the claim.\\
To prove the second part, from It\^o formula we know 
\begin{align}\label{pexpmom}
e^{p\lambda(t)} =e^{p\lambda(0)}+ \int_0^t  e^{p \lambda(s)}p\beta (\lambda_0 -\lambda(s)) ds+  \int_0^t \int_{\mathbb{R}_+} e^{p\lambda(s)} (p\alpha + e^\alpha -1) 1_{(0, \lambda(s)]}(z) \overline{N}(dz, ds).
\end{align}	
Taking expectation for every $p \geq \frac{e^\alpha-1}{\beta -\alpha}$ and applying \eqref{5}, we result 
\begin{align*}
\E(e^{p\lambda(t)} )&= e^{p\lambda(0)}+ p\beta \lambda_0  \int_0^t  \E(e^{p \lambda(s)}) ds +  (p\alpha -p\beta + e^\alpha -1) \int_0^t \E( e^{p\lambda(s)}  \lambda(s)) ds\\
& \leq e^{p\lambda(0)}+ p\beta \lambda_0  \int_0^t  \E(e^{p \lambda(s)}) ds. 
\end{align*}	
The Gronwall inequality and then the Young inequality show that for every $p \geq 1$, 
\begin{equation}\label{suppexp}
\sup_{0 \leq t \leq T} \E(e^{p\lambda(t)} ) < \infty.
\end{equation}
 In addition, Taking expectation from the supremum of \eqref{pexpmom}, and then applying \eqref{5}, and Burkholder-Davis-Gundy inequality deduce
\begin{align*}
\E( \sup_{0 \leq t \leq T}e^{p\lambda(t)} )&\leq  e^{p\lambda(0)}+ p\beta \lambda_0  \int_0^T  \E(\sup_{0 \leq s \leq T} e^{p \lambda(s)}) ds +  (p\alpha + e^\alpha -1)^2 \E\Big(\vert \int_0^T  e^{2p\lambda(s)}  \lambda(s) ds\vert^{\frac12}\Big)\\
& +  (p\alpha + e^\alpha -1) \int_0^T \E( e^{p\lambda(s)}  \lambda(s)) ds\\
& \leq e^{p\lambda(0)}+ p\beta \lambda_0  \int_0^t  \E(\sup_{0 \leq s \leq T}e^{p \lambda(s)}) ds + \frac12  (p\alpha + e^\alpha -1)^2  \int_0^T  \E\Big( e^{2p\lambda(s)}  \lambda(s)\Big) ds \\
&+ \frac12 (p\alpha + e^\alpha -1)^2 + (p\alpha + e^\alpha -1) \int_0^T \E\Big(\frac12 e^{2p\lambda(s)}+ \frac12 \lambda^2(s)\Big) ds.\\
\end{align*}	
Finally, since $\E\Big( e^{2p\lambda(s)}  \lambda(s)\Big) \leq \E\Big( \frac12 e^{4p\lambda(s)} + \frac12\lambda^2(s)\Big) $, from \eqref{suppexp}, and then applying the Gronwall inequality, the proof is completed.
\end{proof}
\section{Hawkes stochastic differential equations}\label{sec4}
In this section, we introduce the model, state the assumptions and consider the properties of the solution in some lemmas we need in the main results. \\
Let $(\Omega,\mathcal{F},\mathit{P})$ be a Wiener-Poisson space with a risk neutral probability $\mathit{P}$.
We assume that the underlying asset price $S=(S_{t})_{t\in[0,T]}$ with the jump stochastic intensity process $\lambda=(\lambda(t))_{t\in[0,T]}$ of Hawkes process $N_{t}$ can be governed by the following system of SDEs:
\begin{equation}\label{6}
\begin{cases}
S_{t}&=S_{0}+\int_{0}^{t}(\mu-(e^{J_s}-1)\lambda_s)S_sds+\int_{0}^{t}\sigma S_sdW_s+\int_{0}^{t}(e^{J_s}-1)S_s{N}(ds)
\\
d\lambda_{t}&=\beta(\lambda_{0}-\lambda_{t})dt+\alpha dN_t,
\end{cases}
\end{equation}
where $(W_{t})_{t\in[0,T]}$ is a Brownian motion, $N$ is independent of $W_t$, $\mu$ denotes the riskless interest rate, $J$ is a measurable function, $\sigma$ is a positive constant. Assume the following conditions throughout the paper.\\
{\bf Condition  H1:} \begin{itemize}
\item Assume that for every $p \ge 1$, we have
$\int_0^T (e^{J_s}-1)^{2p} ds < \infty.$ 
\item The function $J$ is an non-negative strictly increasing function with $J_0=0$. 
\end{itemize}
From \eqref{kN}, the solution to the stochastic differential \eqref{6} is as follows, see \cite{n22}.
\begin{align}\label{7}
\nonumber S_{t}
= &S_0+\int_{0}^{t}\mu S_sds+\int_{0}^{t}\sigma S_sdW_s+\int_{0}^{t}\int_{\mathbb{R}_{+}}(e^{J_s}-1)S_s1_{\left(0,\lambda(s)\right]}(z)[\overline{N}(ds\times dz)-dzds]\\
= &S_0+\int_{0}^{t}\mu S_sds+\int_{0}^{t}\sigma S_sdW_s+\int_{0}^{t}\int_{\mathbb{R}_{+}}(e^{J_s}-1)S_s1_{\left(0,\lambda(s)\right]}(z)\tilde{\overline{N}}(ds\times dz),
\end{align}
where $\tilde{\overline{N}}(dt,dz)=\overline{N}(dt,dz)-dtdz$. According to It\^o formula in
\cite{n25}, the solution of the stochastic differential equation \eqref{7} is in the form:
\begin{align*}
 S_{t}=S_{0}\exp\Big\{(\mu-\frac{1}{2}\sigma^{2})t+\sigma W_{t}&+\int_{0}^{t}\int_{\mathbb{R}_{+}}\{\ln[1+(e^{J_s}-1)1_{\left(0,\lambda(s)\right]}(z)]-(e^{J_s}-1)1_{\left(0,\lambda(s)\right]}(z)\}dzds\\
+&\int_{0}^{t}\int_{\mathbb{R}_{+}}\ln[1+(e^{J_s}-1)1_{\left(0,\lambda(s)\right]}(z)]\tilde{\overline{N}}(ds\times dz)\Big\}.
\end{align*}
Since $1_{\left(0,\lambda(s)\right]}(z)=1$ if and only if $z\in{\left(0,\lambda(s)\right]}$, so
\begin{equation}\label{8}
S_{t}=S_{0}\exp\Big\{(\mu-\frac{1}{2}\sigma^{2})t+\sigma W_{t}
 +\int_{0}^{t}(J_s-e^{J_s}+1)\lambda(s)ds+\int_{0}^{t}\int_{\mathbb{R}_{+}}J_s1_{\left(0,\lambda(s)\right]}(z)\tilde{\overline{N}}(ds\times dz)\Big\}=:S_0e^{X_t}.
\end{equation}
Let  $Y_t$ be the solution of the following SDE 
\begin{equation}\label{9}
		dY_t=Y_t(e^{J_{t}}-1)1_{\left(0,\lambda(s)\right]}(z)\tilde{\overline{N}}(dt,dz), \qquad Y_0=1.
\end{equation} 
From It\^o formula, we have 
\begin{align*}
	  Y_t & =\exp\Big\{\int_{0}^{t}(J_{s}-e^{J_{s}}+1)\lambda(s)ds
		+\int_{0}^{t}\int_{\mathbb{R}_{+}}J_{s}1_{\left(0,\lambda(s)\right]}(z)\tilde{\overline{N}}(ds\times dz)\Big\}\\
		&= \exp\Big\{-\int_{0}^{t}(e^{J_{s}}-1)\lambda(s)ds
		+\int_{0}^{t}\int_{\mathbb{R}_{+}}J_{s}1_{\left(0,\lambda(s)\right]}(z)\overline{N}(ds\times dz)\Big\}.
\end{align*} 
With a similar proof of Theorem 3.1 in \cite{song2022regularity} and Section 5.1.1 of \cite{menaldi2008stochastic}, the equation \eqref{9} has a unique solution.
\subsection{Properties of the solution}\label{sec4.1}
Here, we establish some properties of the solution $S_t$ are needed in the main results.
\begin{lemma}\label{lemma4.1}
Let $X_t$ and $Y_t$ defined in \eqref{8} and \eqref{9}. Under Condition {\bf H1}, for every $p \geq 2$,
\begin{equation*}
\E\Big(Y_t^p\Big) < \infty,  \qquad and \qquad
 \mathbb{E}\Big(\sup_{0 \leq t \leq T}e^{pX_t}\Big)<\infty.
\end{equation*}
\end{lemma}
\begin{proof}
From It\^o formula and Hölder inequality, we have 
\begin{align*}
\sup_{0 \leq t \leq T} \E(Y_t^p) &=  \sup_{0 \leq t \leq T}\E\Big(\exp\Big\{-\int_{0}^{t}p(e^{J_{s}}-1)\lambda(s)ds
		+\int_{0}^{t}\int_{\mathbb{R}_{+}}pJ_{s}1_{\left(0,\lambda(s)\right]}(z)\overline{N}(ds\times dz)\Big\}\Big)\\
		&\leq  \sup_{0 \leq t \leq T}\E\Big(\exp\Big\{-\int_{0}^{t}(e^{2pJ_{s}}-1)\lambda(s)ds
		+\int_{0}^{t}\int_{\mathbb{R}_{+}}2pJ_{s}1_{\left(0,\lambda(s)\right]}(z)\overline{N}(ds\times dz)\Big\}\Big)^\frac12\\
		& \sup_{0 \leq t \leq T}\E\Big(\exp\Big\{\int_{0}^{t}(e^{2pJ_{s}}-1)-2p(e^{J_{s}}-1)\lambda(s)ds\Big\}\Big)^\frac12\\
		& \leq  \sup_{0 \leq t \leq T}\E\Big(\exp\Big\{\int_{0}^{t}e^{2pJ_{s}}\lambda(s)ds\Big\}\Big)^\frac12\\
		& \leq \sup_{0 \leq t \leq T} \E\Big(\exp\Big\{\sup_{0 \leq s \leq T} \lambda(s) \int_{0}^{t}e^{2pJ_{s}}ds\Big\}\Big)^\frac12 <\infty,
\end{align*}
where we used from Lemma \ref{ppp0} and Condition {\bf H1} in the last inequality.\\
To prove the second part, the definition of $X_t$ and its relation to $Y_t$ show that it is sufficient to prove the claim just for $Y_t$. 
To this end, according to Burkholder-Davis-Gundy inequality, there exists some constant $C_p$ such that 
\begin{align}
\mathbb{E}\Big(\sup_{0 \leq s \leq t}\lvert{Y_s}\rvert^{p}\Big)  &\le  Y_0^{p}+C_p\mathbb{E}\Big(\int_{0}^{t}\int_{{\mathbb{R}_+}} Y_s^{2}(e^{J_{s}}-1)^{2}1_{\left(0,\lambda(s)\right]}^{2}(z)dzds\Big)^{\frac{p}{2}}\nonumber\\
		& + C_p\mathbb{E}\Big(\int_{0}^{t}\int_{{\mathbb{R}_+}} Y_s^{p}(e^{J_{s}}-1)^{p}1_{\left(0,\lambda(s)\right]}^{p}(z)dzds\Big)\nonumber\\
		 & \leq  Y_0^{p}+C_p T^p\int_{0}^{t} \mathbb{E}\Big(Y_s^{p}(e^{J_{s}}-1)^{p}\lambda^{\frac{p}{2}}(s)\Big)ds+
		 C_p \int_{0}^{t}\mathbb{E}\Big( Y_s^{p}(e^{J_{s}}-1)^{p}\lambda(s)\Big)ds\nonumber\\
		 &=:  Y_0^{p}+C_p \int_{0}^{t}\Big(T^p I_1(s) +I_2(s)\Big) ds.\label{i1i2int}
\end{align}
Applying Jensen inequality, there exists some constant $c$ that 
\begin{align}
\int_{0}^{t} (T^p I_1(s)+I_2(s))ds
		& \leq  2 T^{p+1} \sup_{0 \leq s \leq t} \mathbb{E}\Big( Y_s^{2p}\Big) + 4 T^p\mathbb{E}\Big(\int_{0}^{t}(e^{J_{s}}-1)^{4p} ds\Big)+  4 T^{p+1}\sup_{0 \leq s \leq t} \mathbb{E}\Big(\lambda^{2p}(s) \Big)\nonumber\\
				&+  2 T \sup_{0 \leq s \leq t} \mathbb{E}\Big( Y_s^{2p}\Big) + 4 \mathbb{E}\Big(\int_{0}^{t}(e^{J_{s}}-1)^{4p} ds\Big)+  4 T\sup_{0 \leq s \leq t} \mathbb{E}\Big(\lambda^{4}(s) \Big).\label{i1i2}
\end{align}
Substitute \eqref{i1i2} into \eqref{i1i2int}, and use the first part of Lemma \ref{ppp0}, and
	Condition ${\bf H1}$ to complete the proof.\\
\end{proof}
We note that assumptions (5.95), (5.96) and (5.98) of Lemma 5.19 in \cite{menaldi2008stochastic} for the coefficients of the SDE $\eqref{7}$ satisfy. Therefore from It\^o formula, 
\begin{align*}
	d\left( \frac{1}{S_t} \right) &= -\frac{\mu}{S_t} dt - \frac{\sigma}{S_t} dW_t + \frac{\sigma^2}{S_t} dt \\
	& + \int_{\mathbb{R}_+} \left[ \frac{1}{S_{t^-}(1 + (e^{J_t}-1)1_{\left(0,\lambda(s)\right]}(z))} - \frac{1}{S_{t^-}} \right] \tilde{\overline{N}}(dt \times dz) \\
	& + \int_{\mathbb{R}_+} \left[ \frac{1}{S_{t^-}(1 + (e^{J_t}-1)1_{\left(0,\lambda(s)\right]}(z))} - \frac{1}{S_{t^-}} + \frac{(e^{J_t}-1)1_{\left(0,\lambda(s)\right]}(z)}{S_{t^-}} \right] \nu(dz)\,dt\\
 &= (-\mu + \sigma^2)S_t^{-1} dt - \sigma S_t^{-1} dW_t\\
	& + \int_{\mathbb{R}_+} S_{t^-}^{-1}(e^{-J_t}-1)1_{\left(0,\lambda(s)\right]}(z)  \tilde{\overline{N}}(dt \times dz) \\
	& + \int_{\mathbb{R}_+}S_{t^-}^{-1} 1_{\left(0,\lambda(s)\right]}(z)(e^{-J_t} + e^{J_t}-2)dzdt.
\end{align*}	
\begin{lemma}\label{yvaroon}
Let $Y_t$ defined in \eqref{9}. Under Condition {\bf H1}, for every $p \geq 1$,
\begin{equation*}
\sup_{0 \leq t \leq T}\E\Big(S_t^{-p}\Big) < \infty.
\end{equation*}
\end{lemma}
\begin{proof}
We will prove it in \autoref{appendixA}.
\end{proof}
In the last part of this section note that getting the partial derivatives of $S_t$ with respect to $S_0$ show that the stochastic flow of $S_t$ exists and it is  
\begin{equation*}
\frac{\partial{S_t}}{\partial{S_0}}=\frac{S_t}{S_0}=\exp\{X_t\}=Y_t \exp\{(\mu-\frac{\sigma^{2}}{2})t+\sigma W_{t}\}. \label{flow1}
\end{equation*}
As a consequence of Lemma \ref{ppp0}, this flow is in $L^{p}$-space for every $p \geq 2$.
\subsection{Weak derivatives of the solution}
Now, we calculate Malliavin derivative of the solution $S_t$ of a Hawkes SDE, accordance with the definitions and lemmas mentioned in \autoref{appendixF},  in witch an essential different as stochastic intensity appears. Based on equation \eqref{8}, thanks to the proof of Theorem $3.1$ of \cite{n20}, and applying Proposition \ref{lipF}, Malliavin derivative of $D_{u,z}^{\overline{N}}S_t$ for every $u<t$ is as follows
\begin{align}
\nonumber D_{u,z}^{\overline{N}}S_t=&D_{u,z}^{\overline{N}}g(X_t)=g(X_t+D_{u,z}^{\overline{N}}X_t)-g(X_t)\\
\nonumber&=S_0e^{X_t+D_{u,t}^{\overline{N}}X_t}-S_0e^{X_t}=S_0e^{X_t}(e^{D_{u,z}^{\overline{N}}X_t}-1)=S_t(e^{D_{u,z}^{\overline{N}}X_t}-1), 
\end{align}
and
\begin{equation}\label{DNX}
 D_{u,z}^{\overline{N}}X_t=J_u1_{\left(0,\lambda(u)\right]}(z)+\int_{u}^{t}k_{(u,z)}(s)N_{(u,z)}(ds)-\int_{u}^{t}(e^{J_s}-1)D_{u,z}^{\overline{N}}\lambda(s)ds,
\end{equation}
where for every $s>u$
\begin{equation}\label{sign}
k_{(u,z)}(s)=sign(\varphi_s(\overline{N}\lvert_{\left(-\infty,s\right)}+\epsilon_{(u,z)})-\varphi_s(\overline{N}\lvert_{\left(-\infty,s\right)}))J_s=sign(D_{u,z}^{\overline{N}}\lambda(s))J_s, 
\end{equation}
\begin{equation}\label{nuz}
N_{(u,z)}(ds)=\overline{N}(ds\times(\varphi_s(\overline{N}\lvert_{\left(-\infty,s\right)}+\epsilon_{(u,z)})\land\varphi_s(\overline{N}\lvert_{\left(-\infty,s\right)}), \varphi_s(\overline{N}\lvert_{\left(-\infty,s\right)}+\epsilon_{(u,z)})\lor\varphi_s(\overline{N}\lvert_{\left(-\infty,s\right)})]).
\end{equation}
Here, $\epsilon_{(u,z)}$ denotes the Dirac measure at $(t,z)\in\mathbb{R}\times\mathbb{R_+}$ and for ease of notation, we denoted by $a\land b$ and $a\lor b$ the minimum and the maximum between $a,b\in\mathbb{R}$, respectively.\\
Additionally, the Malliavin derivative of $\lambda$ is as follows.
\begin{equation}\label{DNlambda}
	D_{u,z}^{\overline{N}}\lambda_{t}=\alpha e^{-\beta(t-u)}1_{\left(0,\lambda(u)\right]}(z)+\int_{u}^{t}\alpha e^{-\beta(t-s)}sign(D_{u,z}^{\overline{N}}\lambda(s))N_{(u,z)}(ds),
\end{equation}
\begin{lemma}\label{dlambdaexp}
For every $p \geq 2$,
\begin{equation*}
\E( \sup_{0 \leq u \leq t \leq T}e^{p D_{u,z}^{\overline{N}} \lambda_t})<  \infty.
\end{equation*}
\end{lemma}
\begin{proof}
See \autoref{appendixE}.
\end{proof}
In the rest of the section, we will present that the inverse of $D_{u,z}^{\overline{N}}S_{t}$ exists and belongs to $L^p(\Omega)$ for every $p \geq 2$.\\
Substitute \eqref{DNlambda} into \eqref{DNX}, and derive 
\begin{align*} 
\nonumber
 D_{u,z}^{\overline{N}}X_t=& J_u1_{\left(0,\lambda(u)\right]}(z)+\int_{u}^{t}k_{(u,z)}(s)N_{(u,z)}(ds)-\int_{u}^{t}(e^{J_s}-1)D_{u,z}^{\overline{N}}\lambda(s)ds\\ \nonumber
=& J_u1_{\left(0,\lambda(u)\right]}(z)+\int_{u}^{t}J_s sign(D_{u,z}^{\overline{N}}\lambda(s))N_{u,z}(ds)\\
		\nonumber
		-&\alpha1_{\left(0,\lambda(u)\right]}(z)\int_{u}^{t}(e^{J_s}-1)e^{-\beta(s-u)}ds -\int_u^t \int_{u}^{s}\alpha (e^{J_s}-1)e^{-\beta(s-l)}sign(D_{u,z}^{\overline{N}}\lambda(l))N_{(u,z)}(dl)ds\\
		\nonumber
		=& \Big(J_u- \alpha \int_{u}^{t}(e^{J_s}-1)e^{-\beta(s-u)}ds \Big)1_{\left(0,\lambda(u)\right]}(z)
		+\int_{u}^{t}J_ssign(D_{u,z}^{\overline{N}}\lambda(s))N_{(u,z)}(ds)\\
		\nonumber
-& \alpha\int_{u}^{t}\Big(\int_{l}^{t}(e^{J_s}-1) e^{-\beta(s-l)}ds  \Big) sign(D_{u,z}^{\overline{N}}\lambda(l)) N_{(u,z)}(dl), 
\end{align*} 
Define $F(v):=J_v-\int_{v}^{t}(e^{J_s}-1)\alpha e^{-\beta(s-v)}ds$, for every $0 \leq v \leq t$, and rewrite $D_{u,z}^{\overline{N}}X_t$ in the following form.
\begin{align}
 D_{u,z}^{\overline{N}}X_t= F(u)1_{\left(0,\lambda(u)\right]}(z)
		+\int_{u}^{t} F(l) sign(D_{u,z}^{\overline{N}}\lambda(l)) N_{(u,z)}(dl). \label{dnxF}
\end{align} 
Now, we will consider some properties of the function $F(.)$ that we need in the following.\\
Define 
$G(v)={e^{J_v}-1}-\beta \int_{v}^t (e^{J_s }-1) ds$ for every $0 \leq v \leq t$. Due to the strictly increasing property of the function $J_.$,  its positivity, and the facts that $G(0)<0$ and $G(t) >0$, we deduce that $G(.)$ is an strictly increasing function with exactly one root, called $v_0$. It is obvious that for every $v > v_0$, $G(v) > 0$.\\
 As a consequence, we show the similar result for the function $F(.)$ in the following lemma. 
 \begin{lemma}\label{v1}
 The function $F: [0, T] \rightarrow \mathbb{R}$ is an increasing function and there exists some $v_1 \in [v_0, t]$ such that  the function $F$ is positive on $[v_1, t]$. 
 \end{lemma}
 \begin{proof}
 Take the derivative of $F$ and result for every $v > v_0$
 \begin{equation*}
 F'(v)=J'_v -\beta e^{\beta v}\int_v^t \alpha e^{-\beta s}(e^{J_s}-1) ds + \alpha (e^{J_v}-1) \geq  J'_v +\alpha G(v)>0.
 \end{equation*}
It shows that $F(.)$ is increasing on $[v_0, t]$. If $F(v_1) \geq 0$, then we put $v_1:=v_0$ and the proof is completed. If  
$F(v_0) <0$, the fact that $F(t) >0$ and increasing property of $F$ deduce the existence some $v_1 \in [v_0, t]$ such that for every $v > v_1$, $F(v)>0$.   
 \end{proof}
\begin{lemma}\label{lemmFDlambda}
There exists some constant $C_0$ such that for any $r\geq 0$, and every $u > v_1$, 
\begin{equation*}
\mathbb{E}\Big(\exp\Big\{-r \int_{u}^{t} F(l) sign(D_{u,z}^{\overline{N}}\lambda(l)) N_{(u,z)}(dl)\Big\}\Big) \le \mathbb{E}\Big(\exp\Big\{TC_0\alpha r1_{[0, \lambda(u)]}(z) 1_{\{F(u) \leq \frac1r\}}J_u\Big\}\Big).
\end{equation*}
\end{lemma}
\begin{proof}
According to the definitions \eqref{nuz} and \eqref{sign}, setting 
$$\mathcal{A}_1:= \varphi_s(\overline{N}\lvert_{\left(-\infty,s\right)}+\epsilon_{(u,z)})\land\varphi_s(\overline{N}\lvert_{\left(-\infty,s\right)}),
\qquad \mathcal{A}_2;= \varphi_s(\overline{N}\lvert_{\left(-\infty,s\right)}+\epsilon_{(u,z)})\lor\varphi_s(\overline{N}\lvert_{\left(-\infty,s\right)}),$$
thanks to the proof of Theorem 3.1 in \cite{n20}, stated that $N_{(u,z)}$ on $(t,\infty)$ has 
$\mathcal{F}_t^{\overline{N}}$-stochastic intensity $\{\vert D_{t,z}\lambda(s)\vert\}_{s \geq t}$, we result 
\begin{align}
\mathbb{E}\Big(\exp\Big\{-r \int_{u}^{t} F(l) & sign(D_{u,z}^{\overline{N}}\lambda(l)) N_{(u,z)}(dl)\Big\}\Big) \nonumber\\
&= \mathbb{E}\Big(\exp\Big\{-r \int_{u}^{t}\int_{\mathbb{R}_+} F(l) sign(D_{u,z}^{\overline{N}}\lambda(l)) 1_{(\mathcal{A}_1, \mathcal{A}_2]}(z_0)\overline{N}(dl\times dz_0)\Big\}\Big)\nonumber\\
&\le 
\mathbb{E}\Big(\exp\Big\{-\int_{u}^{t}(1-e^{-rF(s)})\vert D_{u,z}^{\overline{N}}\lambda(s)\vert ds\Big\}\Big)\nonumber\\
&\leq \mathbb{E}\Big(\exp\Big\{-C_0r\int_{u}^t 1_{\{F(s)\leq\frac{1}{r}\}}F(s) \vert D_{u,z}^{\overline{N}} \lambda(s)\vert ds\Big\}\Big), \label{expr}
\end{align}
where $C_0$ is a constant that $1-e^{-x} \geq C_0 x$, for every $x \in (-1,1)$.
Applying \eqref{DNlambda}, Fubini theorem, and the fact that for every $s \in [u, t]$ we have $F(u) \leq F(s)$, will derive
\begin{align}
	\int_{u}^{t} 1_{\{F(s) \leq \frac1r\}}F(s)e^{-\beta (s-u)}ds&=\int_{u}^{t}1_{\{F(s) \leq \frac1r\}}\Big(J_s e^{-\beta (s-u)}
	-\alpha\int_{s}^{t}(e^{J_l}-1)e^{-\beta(l-u)}dl\Big)ds\nonumber\\
	&=\int_{u}^{t} 1_{\{F(s) \leq \frac1r\}}J_s e^{-\beta (s-u)}ds-\alpha\int_{u}^{t} \int_{u}^{l}1_{\{F(s) \leq \frac1r\}}(e^{J_l}-1)e^{-\beta(l-u)}dsdl\nonumber\\
	 &\geq\int_{u}^{t}1_{\{F(s) \leq \frac1r\}}J_s e^{-\beta (s-u)} ds-1_{\{F(u) \leq \frac1r\}}\alpha\int_{u}^{t}(e^{J_l}-1)e^{-\beta (l-u)}(l-u)dl\nonumber\\
	&\geq \int_{u}^{t} 1_{\{F(s) \leq \frac1r\}}J_se^{-\beta (s-u)}ds-t\alpha 1_{\{F(u) \leq \frac1r\}}\int_{u}^{t} (e^{J_l}-1)e^{-\beta(l-u)}dl\nonumber\\
	&=\int_{u}^{t}1_{\{F(s) \leq \frac1r\}}J_se^{-\beta (s-u)}ds-t1_{\{F(u) \leq \frac1r\}} [J_u-F(u)]. \label{Fbeta}
\end{align}
Substituting \eqref{Fbeta} into \eqref{expr} deduce that 
\begin{align*}
\mathbb{E}\Big(& \exp\Big\{-r \int_{u}^{t} F(l) sign(D_{u,z}^{\overline{N}}\lambda(l)) N_{(u,z)}(dl)\Big\}\Big) \nonumber\\
&\leq \mathbb{E}\Big(\exp\Big\{-C_0\alpha r1_{[0, \lambda(u)]}(z) \int_{u}^t 1_{\{F(s)\leq\frac{1}{r}\}}F(s) e^{-\beta (s-u)}  ds\Big\}\Big) \nonumber\\
&\leq \mathbb{E}\Big(\exp\Big\{-C_0\alpha r1_{[0, \lambda(u)]}(z)\Big(\int_{u}^t 1_{\{F(s) \leq\frac{1}{r}\}} J_s e^{-\beta (s-u)}  ds -t J_u+tF(u)\Big)\Big\}\Big)\nonumber \\
&\leq \mathbb{E}\Big(\exp\Big\{TC_0\alpha r1_{[0, \lambda(u)]}(z) 1_{\{F(u) \leq \frac1r\}}J_u\Big\}\Big). 
\end{align*}
\end{proof}
\begin{lemma}\label{varoon}
	For any $p\geq1$, and $(u,z) \in [0, T] \times \mathbb{R}_+$ with $u > v_1$ such that $P(\lambda (u) > z) > 0$, 
\begin{equation*}
	\nonumber 
\mathbb{E}\Big(\frac{1_{\left(0,\lambda(u)\right]}(z)}{\vert D_{u,z}^{\overline{N}}X_t\vert^{p}}\Big)<\infty.
\end{equation*}	
\end{lemma}
\begin{proof}
Our proof is motivated by the proof of Theorem $1.2$ in \cite{song2022regularity}. From the definition of Gamma function, and Lemma \ref{lemmFDlambda} we have 
\begin{align*}
		\nonumber 
&\mathbb{E}\Big(\frac{1_{\left(0,\lambda(u)\right]}(z)}{(D_{u,z}^{\overline{N}}X_t)^{p}}\Big)\\
		&=\frac{1}{\Gamma(p)}\mathbb{E}\Big(\int_{0}^{\infty}r^{p-1}1_{\left(0,\lambda(u)\right]}(z)\exp\Big\{-rF(u) 1_{\left(0,\lambda(u)\right]}(z)ds
		-r \int_{u}^{t} F(l) sign(D_{u,z}^{\overline{N}}\lambda(l)) N_{(u,z)}(dl)\Big\}dr\Big)\\
	&=\frac{1}{\Gamma(p)}\Big(\int_{0}^{\infty}r^{p-1}\mathbb{E}^{\frac{1}{2}}\Big(1_{\left(0,\lambda(u)\right]}(z)e^{-2rF(u)1_{\left(0,\lambda(u)\right]}(z)}\Big)
 \mathbb{E}^{\frac{1}{2}}\Big(\exp\Big\{2TC_0\alpha r1_{[0, \lambda(u)]}(z)1_{\{F(u) \leq \frac1r\}}J_u\Big\}dr\Big)\\
& \leq  \frac{1}{\Gamma(p)}\Big(\int_{0}^{\infty}r^{p-1}P(\lambda(u) >z )e^{-rF(u)}
 e^{TC_0\alpha r1_{\{F(u) \leq \frac1r\}}J_u}dr\Big)\\
 & \leq  P(\lambda(u) >z )\frac{1}{\Gamma(p)}\int_{0}^{\infty}r^{p-1}e^{-rF(u)}[e^{TC_0\alpha \frac{J_u}{F(u)})}]dr=P(\lambda(u) >z )\frac{1}{F(u)^p }[e^{TC_0\alpha \frac{J_u}{F(u)})}] < \infty.
\end{align*}
\end{proof}
\section{Delta of options derived by the solution of Hawkes SDEs}
In this section, we find the Wiener-Poisson weight of calculating  the delta of the following two types of options; European options, Asian options; when the underlying assets are driven by Hawkes processes.\\
First, from \eqref{kN}, we can write that 
\begin{equation*}
	\delta^{N}(u)=\int_{0}^{T}u(t)(N(dt)-\lambda(t)dt)=\int_{0}^{T}\int_{{\mathbb{R}_+}}u(t)1_{(0,\lambda(t)]}(z)(\overline{N}(dt,dz)-dzdt).
\end{equation*}
By placing $\nu(dz)=1_{(0,\lambda(t)]}(z)dzdt$ in \autoref{appendixF}, we result the duality relation as follows.
\begin{align*}
	\mathbb{E}\Big(F\delta^{N}(u)\Big)&=\mathbb{E}\Big(F\int_{0}^{T}\int_{{\mathbb{R}_+}}u(t)1_{(0,\lambda(t)]}(z)(\overline{N}(dt,dz)-dzdt)\Big)\\
	&=\mathbb{E}\Big(\int_{0}^{T}\int_{{\mathbb{R}_+}}D^{\overline{N}}_{t,z}Fu(t,z)1_{(0,\lambda(t)]}(z)dzdt\Big)\\
	&=\mathbb{E}\Big(\langle D^{\overline{N}}_{t,z}F,u(t,z)\rangle _{L^2([0,T]\times \mathbb{R}_0)}\Big).
\end{align*}
In Poisson approach, from Proposition \ref{ddelta}, the sensitivity analysis can be represented in the following lemma. 
\begin{lemma}\label{deltann}
If there is some function $u(.,.)$ such that  $1_{\left(0,\lambda(t)\right]}(z)u(t,z)$ belongs to the domain of $\delta^{\overline{N}}$ and 
	\begin{align}\label{deltaN}
	\mathbb{E}\Big(f^{\prime}(S_T)\frac{\partial{S_T}}{\partial{S_0}}\Big)=\mathbb{E}\Big(\int_{0}^{T}\int_{\mathbb{R}_+}1_{\left(0,\lambda(t)\right]}(z)u(t,z)\{f(S_T+D_{t,z}^{\overline{N}}S_T)-f(S_T)\}dzdt\Big),
	\end{align}
	then $\frac{\partial }{\partial S_0}\mathbb{E}\Big(f(S_T)\Big) =\mathbb{E}\Big(f(S_T)\delta^{\overline{N}}(1_{\left(0,\lambda(t)\right]}(z)u(t,z))\Big)$.
\end{lemma}
As an application, we may calculate the delta in the Wiener-Poisson space for the following examples desired in \cite{huehne}, when $f$  is considered as a payoff function and $\Delta^{\overline{N}}:=\frac{\partial }{\partial S_0}\mathbb{E}\Big(f(S_T)\Big)$.
\subsection{Delta Greek for European options}
Consider the payoff function of a European option as $f(S_T)=(S_T-K)^{+}$, for a strike price $K$. Then from Equation (1.42) in \cite{nualart}, for every $h \in L^2[0, T]$ such that $\int_{0}^{T}D_s^{W}S_Tds \neq 0$, by using a classical approach, 
	\begin{align}
		\nonumber \Delta^{W}=&\frac{\partial}{\partial{S_0}}\mathbb{E}(f(S_T)1_{\overline{N}=0})=\mathbb{E}(1_{\overline{N}=0}f^{\prime}(S_T)\frac{\partial{S_T}}{\partial{S_0}})=\frac{1}{S_0}\mathbb{E}(1_{\overline{N}=0}f^{\prime}(S_T)S_T)\\
	\nonumber	&=\frac{1}{S_0}\mathbb{E}\Big(1_{\overline{N}=0}f(S_T)\delta^W\Big(\frac{h(.) S_T}{\int_{0}^{T}h(s)D_s^{W}S_Tds}\Big)\Big).
	\end{align}
Setting $h\equiv1$ and by using $D_u^{W}S_T=\sigma S_T$, we result
\begin{equation}\label{dwweight}
	\Delta^{W}=\frac{1}{S_0T\sigma}\mathbb{E}\Big(1_{\overline{N}=0}f(S_T)W_T\Big).
\end{equation}
Define $G_{t,z}=D^{\overline{N}}_{t,z}S_T+{S_T-K}=S_Te^{D^{\overline{N}}_{s,z}X_T}-K$. According to Lemma \rm{\ref{v1}} and the equation \eqref{dnxF}, we know if $S_T >K$, then $S_T e^{D_{t,z}^{\overline{N}}X_T}> K$  on $t \in [v_1, T]$, and $S_T e^{D_{t,z}^{\overline{N}}X_T} <  K$ can be happen when $t \in [0,v_1)$. Denote $B(w):= \Big\{(t,z)\in [0, T] \times \mathbb{R}_+;\quad  S_T(w) e^{D_{t,z}^{\overline{N}}X_T(w)} > K\Big\}$, and set $\mathcal{B}_1:= {\int\int}_{B} 1_{\left(0,\lambda(s)\right]}(z_0)  dz_0ds$, and $\mathcal{B}_2:= {\int\int}_{B^c} 1_{\left(0,\lambda(s)\right]}(z_0)  dz_0ds$. We know that $ \mathcal{B}_1+\mathcal{B}_2= \int_0^T \lambda(t) dt$. Now, we can define the function $u$ as follows.
\begin{align}
	u_1(t,z)=&\begin{cases}
			\frac{H_K(S_T)(1+\frac{K}{S_T})}{2S_0\mathcal{B}_1 (exp\{D_{t,z}^{\overline{N}}X_T\}-1)}  \;& (t,z) \in B \cap ([v_1, T] \times  \mathbb{R}_0),\\\\
			- \frac{H_K(S_T)}{2S_0\mathcal{B}_2}  \;&  (t,z) \in B^c  \cap ([0, v_1] \times  \mathbb{R}_0), 
		\end{cases} \label{udeltaN}
\end{align}
where $H_y(x)=1_{x > y}$ is the Heaviside function and $1_A$ is the indicator function of the set $A$. Substituting \eqref{udeltaN} in Equation \eqref{deltaN}, it justifies that $u$ is well defined, then Lemma \ref{deltann} can apply and 
\begin{equation}\label{dNweight}
\Delta^{\overline{N}}=\mathbb{E}\Big(f(S_T)\delta^{\overline{N}}(1_{\left(0,\lambda(t)\right]}(z)u_1(t,z))\Big),
\end{equation}
if $u$ belongs to $Dom(\delta^{\overline{N}})$. This will be shown in the following lemma and be proved in \autoref{appendixB}. To sake of notation, denote 
	$$\mathcal{H}_1(t,z):= \frac{(1+\frac{K}{S_T})H_K(S_T)}{\mathcal{B}_1 (exp\{D_{t,z}^{\overline{N}}X_T\}-1)}, \qquad \mathcal{H}_2(t,z):= \frac{H_K(S_T)}{\mathcal{B}_2}.$$
\begin{lemma}\label{deltadomain}
Two processes $1_{(0, \lambda(t)]}(z)1_{B \cap ([v_1, T] \times  \mathbb{R}_0)}\mathcal{H}_1$ and $1_{(0, \lambda(t)]}(z)1_{B^c  \cap ([0, v_1] \times  \mathbb{R}_0)}\mathcal{H}_2$ belong to $Dom(\delta^{\overline{N}})$.
\end{lemma}
\begin{proof}
We will prove the claim in \autoref{appendixB}.
\end{proof}
\subsection{Delta Greek for Simple Asian options}
Consider a simple Asian options with the payoff function $F=f(Y_T)=(Y_T-K)^{+}$, where
\begin{equation*}\label{equasian}
	Y_T=\frac{1}{T}\int_{0}^{T}S_tdt.
\end{equation*}
Recall Section 2.2 in \cite{nualart}, put $h(t) = S_t$, and use the relation $\delta^W(Gu)=G\delta^W(u)-\int_{0}^{T}D_t^{W}G u_tdt$, with $u(.)=S_.$ and $G=\frac{1}{Y_T}$ to derive 
\begin{align*}
	H:=\delta^W\Big(\frac{Y_TS_.}{\int_{0}^{T}S_u D_u^{W}Y_Tdu}\Big)& =\delta^W\Big(\frac{Y_TS_.}{\frac{\sigma}{T}\int_{0}^{T} S_u \int_{u}^{T}S_tdtdu}\Big)=\delta^W\Big(\frac{Y_TS_.}{\frac{\sigma}{2T}(\int_{0}^{T}S_udu)^2}\Big)=2\delta^W\Big(\frac{S_.}{T\sigma Y_T }\Big) \notag\\
	&= \frac{2}{T\sigma Y_T}\delta^W(S_.)+ \frac{2}{T \sigma}\int_{0}^{T}\frac{D^W_tY_T }{Y_T^2}S_tdt\notag\\
	&=\frac{2}{T\sigma Y_T}\delta^W(S_.)+ T.  
\end{align*}
Then, by a classical computation, we result
\begin{align}
\Delta^{W}=&\frac{\partial}{\partial{S_0}}\mathbb{E}(1_{{\overline{N}}=0}f(Y_T)_=\mathbb{E}(1_{{\overline{N}}=0}f^{\prime}(Y_T)\frac{\partial{Y_T}}{\partial{S_0}})=\frac{1}{S_0}\mathbb{E}(1_{{\overline{N}}=0}f^{\prime}(Y_T)Y_T)\notag\\
&=\frac{1}{S_0}\mathbb{E}\Big(f(Y_T)\delta^W\Big(1_{{\overline{N}}=0}\frac{Y_Th(.)}{\int_{0}^{T}h(u)D_u^{W}Y_Tdu}\Big)\Big)\notag\\
& = \frac{1}{S_0}\mathbb{E}\Big(f(Y_T)1_{{\overline{N}}=0}[\frac{2}{T\sigma Y_T}\delta^W( S_.)+ T]\Big). \label{dwweightAsian}
\end{align}
\begin{lemma}\label{asian1}
If there is some function $u(.,.)$ such that $v(t,z)= 1_{\left(0,\lambda(t)\right]}(z)u(t,z)$ belongs to the domain of $\delta^{\overline{N}}$ and 
	\begin{align*}
		\mathbb{E}(f^{\prime}(Y_T)\frac{\partial{Y_T}}{\partial{S_0}})=\mathbb{E}\Big(\int_{0}^{T}\int_{\mathbb{R}_{+}}1_{\left(0,\lambda(t)\right]}(z)u(t,z)\{f(Y_T+D_{t,z}^{\overline{N}}Y_T)-f(Y_T)\}dzdt\Big),
	\end{align*}
	then $\Delta^{\overline{N}}:=\frac{\partial }{\partial S_0}\mathbb{E}\Big(f(Y_T)\Big) =\mathbb{E}\Big(f(Y_T)\delta^{\overline{N}}(1_{\left(0,\lambda(t)\right]}(z)u(t,z))\Big)$.
\end{lemma}
Denote $C_Y(w):= \Big\{(t,z)\in [0, T] \times \mathbb{R}_+;\quad   D^{\overline{N}}_{t,z}Y_T+{Y_T-K}> 0\Big\}$, and set $\mathcal{C}_1:= {\int\int}_{C_Y} 1_{\left(0,\lambda(s)\right]}(z_0)  dz_0ds$,  and $\mathcal{C}_2:= {\int\int}_{(C_Y)^c} 1_{\left(0,\lambda(s)\right]}(z_0)  dz_0ds$. We know that $ \mathcal{C}_1+\mathcal{C}_2= \int_0^T \lambda(t) dt$. Since $$ f(Y_T+D_{t,z}^{\overline{N}}Y_T)-f(Y_T) = 1_{C_Y}(t,z) \frac{1}{T} \int_t^T S_u (e^{D_{t,z}^{\overline{N}}X_u}-1) du  +1_{(C_Y)^c}(t,z) (K-Y_T),$$ we can define the function $u$ as follows. To sake of notation, denote 
$$\mathcal{H}_3(t,z):= \frac{T(K+Y_T) H_K(Y_T)}{\mathcal{C}_1 \int_t^T S_u (exp\{D_{t,z}^{\overline{N}}X_u\}-1)du}, \qquad \mathcal{H}_4(t,z):= \frac{H_K(Y_T)}{\mathcal{C}_2},$$
and define
\begin{align*}
	u_2(t,z)=&\begin{cases}
		\frac{1}{2S_0}\mathcal{H}_3(t,z)  \;& (t,z) \in C_Y \cap ([v_1, T] \times  \mathbb{R}_0),\\\\
		-\frac{1}{2S_0}\mathcal{H}_4(t,z)  \;& (t,z) \in (C_Y)^c \cap ([0, v_1] \times  \mathbb{R}_0).
	\end{cases}
\end{align*}
The condition of Lemma \ref{asian1} holds for $u_2$, if $u_2 \in Dom(\delta^N)$, and therefore 
\begin{equation}\label{dNweightAsian}
 \Delta^{\overline{N}}=\mathbb{E}\Big(f(Y_T)\delta^{\overline{N}}(1_{\left(0,\lambda(t)\right]}(z) u_2 )\Big). 
 \end{equation}
\begin{lemma}\label{deltadomainY}
Two processes $1_{\left(0,\lambda(t)\right]}(z)1_{C_Y \cap ([v_1, T] \times  \mathbb{R}_0)}\mathcal{H}_3$ and $1_{\left(0,\lambda(t)\right]}(z)1_{(C_Y)^c  \cap ([0, v_1] \times  \mathbb{R}_0)}\mathcal{H}_4$ belong to $Dom(\delta^{\overline{N}})$.
\end{lemma}
\begin{proof}
We will prove the claim in \autoref{appendixC}.
\end{proof}
\section{The convergence of the numerical scheme}
In the numerical experiments, the convergence of the numerical scheme used to discretize the processes and apply them in the computation of the delta for numerical examples is essentially required. So we justify the convergence of our method to the exact solution, first.\\ 
Using \cite{n6}, consider equidistant partition of the interval $[0,T]$: $t_i=t_i(n)=\frac{iT}{n}, i = 0, 1, 2, ..., n$, for any $n\in \mathbb{N}$, and define
the discretized process of X, corresponds to the  given partition, in the following form
\begin{align*}
X_{t_j}^{n}&=X_0+(\mu-\frac{\sigma^{2}}{2})t_j+\sigma W_{t_j}+\int_{0}^{t_j}\int_{\mathbb{R}_{+}}(1-e^{J^n_{s}})1_{\left(0,\lambda^{n}(s)\right]}(z)dzds+\int_{0}^{t_j}\int_{\mathbb{R}_{+}}J^n_{s}N^n(ds,dz)\\
&=X_0+(\mu-\frac{\sigma^{2}}{2})t_j+\sigma W_{t_j}+\int_{0}^{t_j}\int_{\mathbb{R}_{+}}(1-e^{J^n_{s}})1_{\left(0,\lambda^{n}(s)\right]}(z)dzds+\int_{0}^{t_j}\int_{\mathbb{R}_{+}}J^n_{s}1_{\left(0,\lambda^{n}(s)\right]}(z)\overline{N}(ds,dz),
\end{align*}
where for every $s\in[t_i,t_{i+1}]$, we have $J^n_{s}=J_{t_i}$, and 
\begin{equation}\label{lamlam}
\lambda^{n}(s)=\lambda^{n}({t_i})+ \int_{t_i}^s \beta (\lambda_0 - \lambda^n(t_i)) du + \alpha \int_{t_i}^s \int_{\mathbb{R}_+} 1_{(0, \lambda^n({t_i})] }(z)  \overline{N}(dt,dz).
\end{equation} 
Here, ${N}^n(dt,dz)$ is a Poisson process associated to $\overline{N}$ with stochastic intensity $\lambda^n$. In this sense, ${N} -{N}^n$ is a Poisson process with stochastic intensity $\lambda-\lambda^n$. Finally, we define $S_{t_j}^{n}=\exp\{X_{t_j}^{n}\}$.\\
In this sense, for every $t_{i+1}$,
$$\lambda^{n}(t_{i+1})=\lambda^{n}({t_i})+  \beta\Delta t  (\lambda_0 - \lambda^n(t_i)) + \alpha \int_{t_i}^{t_{i+1}} \int_{\mathbb{R}_+} 1_{(0, \lambda^n({t_i})] }(z)  \overline{N}(dt,dz).$$ 
Take the expectation and result 
\begin{align*}
\E(\lambda^n(t_{i+1})) &=\E(\lamn(t_{i})) + \beta \Delta t  \Big(\lambda_0 - \E(\lambda^n(t_i))\Big) + \alpha \E(\lambda^n({t_i})) \Delta t \\
& =\beta \Delta t  \lambda_0 + [1+ (\alpha-\beta) \Delta t ] \E(\lambda^n(t_i))  \leq \beta \Delta t  \lambda_0 + \E(\lambda^n(t_i)).
\end{align*}
Hence, for every $i=1, \cdots, n$, 
\begin{equation}\label{supexpectlambda}
\sup_{1 \leq i \leq n}\E(\lambda^n(t_{i})) \leq \sup_{1 \leq i \leq n}\lambda_0(1+(i-1)\beta \Delta t) < \lambda_0(1+\beta).
\end{equation}
In addition, taking the expectation from \ref{lamlam} and then applying \ref{supexpectlambda} deduce for every $s\in[t_i,t_{i+1}]$
\begin{align}
\E(\vert \lambda^n(s)-\lamn(t_{i})\vert )=\alpha \E\Big(\int_{t_i}^s \int_{\mathbb{R}_+}1_{(0, \lambda^n({t_i})] }(z) dz dt\Big)\leq \alpha \Delta t \sup_{1 \leq i \leq n}\E(\lambda^n(t_{i})) \leq \alpha\lambda_0(1+\beta) \Delta t, \label{pmomlamn}
\end{align}
and for every $p\geq 2$, from Burkholder-Davis-Gundy inequality,
\begin{align}
\E(\vert \lambda^n(s)-\lamn(t_{i})\vert^p )&\leq \alpha^p \Big[\E\Big(\int_{t_i}^s \int_{\mathbb{R}_+}1_{(0, \lambda^n({t_i})] }(z) dz dt\Big)+\E\Big(\int_{t_i}^s \int_{\mathbb{R}_+}1_{(0, \lambda^n({t_i})] }(z) dz dt\Big)^{\frac{p}{2}}\Big]\notag\\
&\leq \alpha^{p} \Delta t \sup_{1 \leq i \leq n}\E(\lambda^n(t_{i}))+\alpha^p (\Delta t)^{\frac{p}{2}-1} \E\Big(\int_{t_i}^s \int_{\mathbb{R}_+}1_{(0, \lambda^n({t_i})] }(z) dz dt\Big) \notag\\
&\leq \alpha^p \lambda_0(1+\beta) \Big(\Delta t+ (\Delta t)^{\frac{p}{2}}\Big) \leq \alpha^p \lambda_0(1+\beta) \Delta t. \label{pmomlamnp}
\end{align}
\begin{lemma}\label{lemlambdan}
For every $p \geq 2$, there exists some constant $c_{1,p}$ such that for every $i=1, ..., n$
$$\sup_{{t_i} \leq t \leq t_{i+1}} \mathbb{E}\Big(\lambda(t) -\lambda^{n}(t)\Big)^{p}\leq c_{1,p} n^{-1}.$$ 
\end{lemma}
\begin{proof}
From the definition of $\lambda^n$, for every $i=1, ..., n$, and then applying Burkholder-Davis-Gundy inequality we deduce 
\begin{align*}
 \mathbb{E}\Big(\vert \lambda(t)- \lambda^n(t)  \vert^p\Big) & \leq  2^{p-1} \mathbb{E}\Big(\vert \lambda(t_i )- \lambda^n(t_i)  \vert^p + \beta^p \int_{t_i}^t \vert \lambda(s) - \lambda^n({t_i}) \vert^p ds \\
&+ \alpha^p \Big\vert \int_{t_i}^t \int_{\mathbb{R}_+} \Big(1_{(0, \lambda(s)] }(z) -1_{(0, \lambda^n({t_i})] }(z) \Big) \overline{N}(ds,dz) \Big\vert^p  \Big)\\
& \leq   2^{p-1}\Big[\mathbb{E}\Big(\vert \lambda(t_i )- \lambda^n(t_i)  \vert^p \Big) + \beta^p \int_{t_i}^t \mathbb{E}\Big(\vert \lambda(s) - \lambda^n({t_i}) \vert^p \Big) ds \\
&+ \alpha^p \int_{t_i}^t \int_{\mathbb{R}_+} \mathbb{E}\Big(1_{(0, \lambda(s)] }(z) -1_{(0, \lambda^n({t_i})] }(z) \Big)^p dzds \\
&+ \alpha^p \mathbb{E}\Big(\Big\vert \int_{t_i}^t \int_{\mathbb{R}_+} \Big(1_{(0, \lambda(s)] }(z) -1_{(0, \lambda^n({t_i})] }(z) \Big)^2 dzds \Big\vert^{\frac{p}{2}}  \Big) \Big].\\
\end{align*}
It is easy to see that $\Big(1_{(0, \lambda(s)] }(z) -1_{(0, \lambda^n({t_i})] }(z) \Big)^p \leq 2^{p} \Big(1_{(0, \lambda(s)] }(z) -1_{(0, \lambda^n({t_i})] }(z) \Big)^2$, and $$\int_{\mathbb{R}_+} \Big(1_{(0, \lambda(s)] }(z) -1_{(0, \lambda^n({t_i})] }(z) \Big)^2 dz \leq \vert \lambda(s) - \lambda^n(t_i) \vert. $$  Hence, applying Hölder inequality results
\begin{align*}
 \mathbb{E}\Big(\vert \lambda(t)- \lambda^n(t)  \vert^p\Big) & \leq
2^{p-1}\Big[\mathbb{E}\Big(\vert \lambda(t_i )- \lambda^n(t_i)  \vert^p \Big) + \beta^p \int_{t_i}^t \mathbb{E}\Big(\vert \lambda(s) - \lambda^n({t_i}) \vert^p \Big) ds \\
&+ 2^{p}\alpha^p \int_{t_i}^t \mathbb{E}\Big(\lambda(s)-\lambda^n({t_i})\Big) ds+ \alpha^p (\triangle t)^{\frac{p}{2}-1} \mathbb{E}\Big(\int_{t_i}^t \vert \lambda(s)-\lambda^n({t_i}) \vert^{\frac{p}{2}}ds\Big)
\Big]\\
& \leq 2^{p-1}\Big[\mathbb{E}\Big(\vert \lambda(t_i )- \lambda^n(t_i)  \vert^p \Big) + 2^p\beta^p \int_{t_i}^t \mathbb{E}\Big(\vert \lambda(s) - \lambda^n({s}) \vert^p \Big) ds \\
&+ 2^p\beta^p \int_{t_i}^t \mathbb{E}\Big(\vert \lambda^n(s) - \lambda^n({t_i}) \vert^p \Big) ds + 2^p\alpha^p \int_{t_i}^t \mathbb{E}\Big(\lambda(s)-\lambda^n(s)\Big) ds\\
&+ 2^p\alpha^p \int_{t_i}^t \mathbb{E}\Big(\lambda^n(s)-\lambda^n({t_i})\Big) ds+\frac12 \alpha^p (\triangle t)^{\frac{p}{2}-1} \mathbb{E}\Big(\int_{t_i}^t \vert \lambda(s)-\lambda^n(s) \vert^{p}ds\Big)\\
&+ \frac12\alpha^p (\triangle t)^{\frac{p}{2}-1} \mathbb{E}\Big(\int_{t_i}^t \vert \lambda^n(s)-\lambda^n({t_i}) \vert^{p}ds\Big) + \alpha^p (\triangle t)^{\frac{p}{2}}
\Big].
\end{align*}
Now, substitute the upper bounds of \ref{pmomlamn} and \ref{pmomlamnp} in the above inequality to result 
\begin{align*}
  \mathbb{E}\Big(\vert \lambda(t)- \lambda^n(t)  \vert^p\Big) & \leq
2^{p-1}\Big[\mathbb{E}\Big(\vert \lambda(t_i )- \lambda^n(t_i)  \vert^p \Big) + \Big( 2^p\beta^p +\frac12\alpha^p (\triangle t)^{\frac{p}{2}-1} \Big)\int_{t_i}^t \mathbb{E}\Big(\vert \lambda(s) - \lambda^n({s}) \vert^p \Big) ds \\
&+ 2^p\alpha^p \int_{t_i}^t \mathbb{E}\Big(\lambda(s)-\lambda^n(s)\Big) ds +C_{p,1} \Delta t\Big]\\
&\leq 2^{p-1}\Big[\mathbb{E}\Big(\vert \lambda(t_i )- \lambda^n(t_i)  \vert^p \Big) + \Big( 2^p\beta^p +\frac12\alpha^p T^p \Big)\int_{t_i}^t \mathbb{E}\Big(\vert \lambda(s) - \lambda^n({s}) \vert^p \Big) ds \\
&+ \frac{2^p}{p}\alpha^p \int_{t_i}^t \mathbb{E}\Big(\vert\lambda(s)-\lambda^n(s)\vert^p\Big) ds + \frac{2^p}{q}\alpha^p  \Delta t +C_{p,1} \Delta t\Big],\\
\end{align*}
where $C_{p,1}$ is a constant and we apply Hölder inequality in the last inequality. Finally, Gronwall inequality deduce 
\begin{align*}
  \mathbb{E}\Big(\vert \lambda(t)- \lambda^n(t)  \vert^p\Big) & \leq
2^{p-1}\Big[\mathbb{E}\Big(\vert \lambda(t_i )- \lambda^n(t_i)  \vert^p \Big)+ \frac{2^p}{q}\alpha^p  \Delta t +C_{p,1} \Delta t \Big] e^{A_{p,1}\Delta t},
\end{align*}
where $A_{p,1}= 2^{p+1}\beta^p + 2^{p-2}\alpha^p T^p +  \frac{2^{2p}}{p}\alpha^p$. We can continue this process recursively to derive there exists some constant $c_{1,p}$ such that 
\begin{align*}
  \mathbb{E}\Big(\vert \lambda(t)- \lambda^n(t)  \vert^p\Big) & \leq (\lambda_0-\lambda^n_0)e^{A_{p,1}T} +\Delta t (e^T-1) \leq c_{1,p} \Delta t.
\end{align*}
\end{proof}
The following two inequalities are required to establish the convergence rate. For every $x,y\in\mathbb{R}$ and $p\in\mathbb{N}$,
\begin{equation}\label{e}
	\vert e^{x}-e^{y}\vert\leq(e^{x}+e^{y})\vert x-y\vert,
\end{equation}
\begin{equation}\label{2p}
	(x+y)^{2p}\leq 2^{2p-1}(x^{2p}+y^{2p}).
\end{equation}
We prove the convergence of $X_t^n$ and $S_t^n$ in the following lemma.
\begin{lemma}\label{lemma1}
Assume that the function $J_.$ is a Lipschitz continuous function. Then, there exists some constant $c_2>0$ such that
	\begin{equation}\label{equboundX}
		\mathbb{E}(X_T-X_T^{n})^{2p}\leq c_2n^{-\frac{p}{2}}, \qquad  \qquad \qquad  \mathbb{E}(S_T-S_T^{n})^{2p}\leq c_2n^{-\frac{p}{2}}.
	\end{equation}
\end{lemma}
\begin{proof} 
Let's start with \eqref{equboundX}. Using \eqref{2p}
	\begin{align}
	\mathbb{E}\Big( X_T-X_T^{n}\Big)^{2p}&=\mathbb{E}\Big(\int_{0}^{T}\int_{\mathbb{R}_{+}}(1-e^{J_s})1_{\left(0,\lambda(s)\right]}(z)dzds+\int_{0}^{T}\int_{\mathbb{R}_{+}}J_s1_{\left(0,\lambda(s)\right]}(z)\overline{N}(ds,dz)\nonumber\\
	&-[\int_{0}^{T}\int_{\mathbb{R}_{+}}(1-e^{J_s^{n}})1_{\left(0,\lambda^{n}(s)\right]}(z)dzds+\int_{0}^{T}\int_{\mathbb{R}_{+}}J_s^{n}1_{\left(0,\lambda^{n}(s)\right]}(z)\overline{N}(ds,dz)]
	\Big)^{2p}\nonumber\\
	&\leq 4^{2p-1}\mathbb{E}\Big( \int_{0}^{T}(\lambda(s)-\lambda^{n}(s))ds\Big)^{2p}+4^{2p-1}\mathbb{E}\Big(\int_{0}^{T}(e^{J_s^{n}}\lambda^{n}(s)-e^{J_s}\lambda(s))ds\Big)^{2p}\nonumber\\
	&+4^{2p-1}\mathbb{E}\Big(\int_{0}^{T}\int_{\mathbb{R}_{+}}[J_s1_{\left(0,\lambda(s)\right]}(z)-J_s^{n}1_{\left(0,\lambda^{n}(s)\right]}(z)]\overline{N}(ds,dz)\Big)^{2p}\nonumber\\
	&\leq 4^{2p-1}\Big(T^2\sup_{0 \leq s \leq T}\mathbb{E}\Big((\lambda(s)-\lambda(s)^{n})^{2p}\Big)+I_1+I_2\Big).\label{xxn}
\end{align}
To bound $I_1$, from the equations \eqref{e}, \eqref{2p}, and Hölder inequality, connection with Lemma \ref{lemlambdan} and Condition {\bf{H1}}, and finally Lipschitz property of $J_.$, we result there exists some constant $c_6$ that 
\begin{align*}
I_1 &\leq 2^{2p-1}\mathbb{E}\Big(\int_{0}^{T}(e^{J_s^{n}}-e^{J_s})\lambda^n(s)ds\Big)^{2p}+2^{2p-1}\mathbb{E}\Big(\int_{0}^{T}e^{J_s}(\lambda^{n}(s)-\lambda(s))ds\Big)^{2p}\nonumber\\
&\leq T^{2p-1}2^{2p-1}\Big(\sum_{i=1}^{n-1}\int_{t_i}^{t_{i+1}}\Big[2^{2p-1}[(e^{J_{t_i}}-1)^{2p}+(e^{J_s}-1)^{2p}]+2^{2p}\Big] (J_{t_i}- J_s) ds\Big)\sup_{0 \leq s \leq T}\mathbb{E}\Big(\lambda^n(s)\Big)^{2p}\nonumber\\
&+8^{2p-1}\mathbb{E}\Big(\int_{0}^{T}(e^{J_s}-1)^{2p}(\lambda^{n}(s)-\lambda(s))^{2p} ds\Big)+8^{2p-1}\mathbb{E}\Big(\int_{0}^{T}(\lambda^{n}(s)-\lambda(s))^{2p} ds\Big)\nonumber\\
& \leq c_6(\triangle t)^{2p}.	
\end{align*}
Also, from \eqref{2p} we have
\begin{align*}		
I_2 &=\mathbb{E}\Big(\int_{0}^{T}\int_{\mathbb{R}_{+}}J_s1_{\left(0,\lambda(s)\right]}(z)\overline{N}(ds,dz)-\int_{0}^{T}\int_{\mathbb{R}_{+}}J_s1_{\left(0,\lambda^{n}(s)\right]}(z)\overline{N}(ds,dz)\\
&+\int_{0}^{T}\int_{\mathbb{R}_{+}}J_s1_{\left(0,\lambda^{n}(s)\right]}(z)\overline{N}(ds,dz)-\int_{0}^{T}\int_{\mathbb{R}_{+}}J_s^{n}1_{\left(0,\lambda^{n}(s)\right]}(z)\overline{N}(ds,dz)\Big)^{2p}\\
	&\leq 2^{2p-1}\mathbb{E}\Big(\int_{0}^{T}\int_{\mathbb{R}_{+}}(J_s1_{\left(0,\lambda(s)\right]}(z)-J_s1_{\left(0,\lambda^{n} (s)\right]}(z))^{2p}dzds\Big)\\
&+2^{2p-1}\mathbb{E}\Big(\int_{0}^{T}\int_{\mathbb{R}_{+}}(J_s1_{\left(0,\lambda(s)\right]}(z)-J_s1_{\left(0,\lambda^n(s)\right]}(z))^{2}dzds\Big)^p\\
	&+2^{2p-1}\mathbb{E}\Big(\int_{0}^{T}\int_{\mathbb{R}_{+}}(J_s1_{\left(0,\lambda^{n}(s)\right]}(z)-J_s^{n}1_{\left(0,\lambda^{n}(s)\right]}(z))^{2p}dzds\Big)\\
	&+2^{2p-1}\mathbb{E}\Big(\int_{0}^{T}\int_{\mathbb{R}_{+}}(J_s1_{\left(0,\lambda^{n}(s)\right]}(z)-J_s^n1_{\left(0,\lambda^n(s)\right]}(z))^{2}dzds\Big)^p.
\end{align*}
Now, using Condition {\bf{H1}}, the equation \eqref{supexpectlambda}, and Lemma \ref{lemlambdan}, there exists some constant $c_7$ that
\begin{align*}
I_2	&\leq
	(2\gamma T)^{2p}[\sup_{0 \leq s \leq T}\mathbb{E}(\vert \lambda(s)-\lambda^{n}(s)\vert)]+(2\gamma\triangle t)^{2p}\mathbb{E}(\int_{0}^{T}\lambda^{n}(s) ds)\\
	&+(2\gamma T)^{2}[\sup_{0 \leq s \leq T}\mathbb{E}(\vert \lambda(s)-\lambda^{n}(s)\vert^p)]+(2\gamma\triangle t)^{2}\mathbb{E}(\int_{0}^{T}\lambda^{n}(s)ds)^p\\
& \leq c_7  (\triangle t),	 		
\end{align*}
Substitute the bounds of $I_1$ and $I_2$ into equation \eqref{xxn} to complete the proof.
\end{proof}
\section{Numerical Example}
In this section, we price European and Asian call options as well as the delta Greek of those for Hawkes processes.\\
Consider the SDE \eqref{6} with the parameters $\sigma=0.10$, $\alpha=0.30$, $\beta=0.80$, $\gamma=0.20$, $\mu=0.05$, $T=1$, $\lambda_0=1$, $S_0=5$. The European call option and simple Asian call option with the expiration date $T$ and the strike price $K$, defined as follows. 
\begin{equation*}
	f(S_T)=\max(S_T-K,0),
\qquad 	f(S_T,Y_T)=\max(\frac{1}{T}\int_{0}^{T}S_tdt-K,0).
\end{equation*} 
The specifications of the computer system with which the program is implemented are Intel(R) Core i$7-9700$K CPU and $64$ GB Memory.\\
Figure \ref{figure1} shows the price of European and Asian call options with $h=0.010$, the number of simulated paths $1000$, and $K=S_0 \times u$ which $u=0.05, 0.10, 0.15, \ldots, 1.20, 1.25, 1.30$.
\begin{figure}[H]
	\centering
	\includegraphics[scale=0.35]{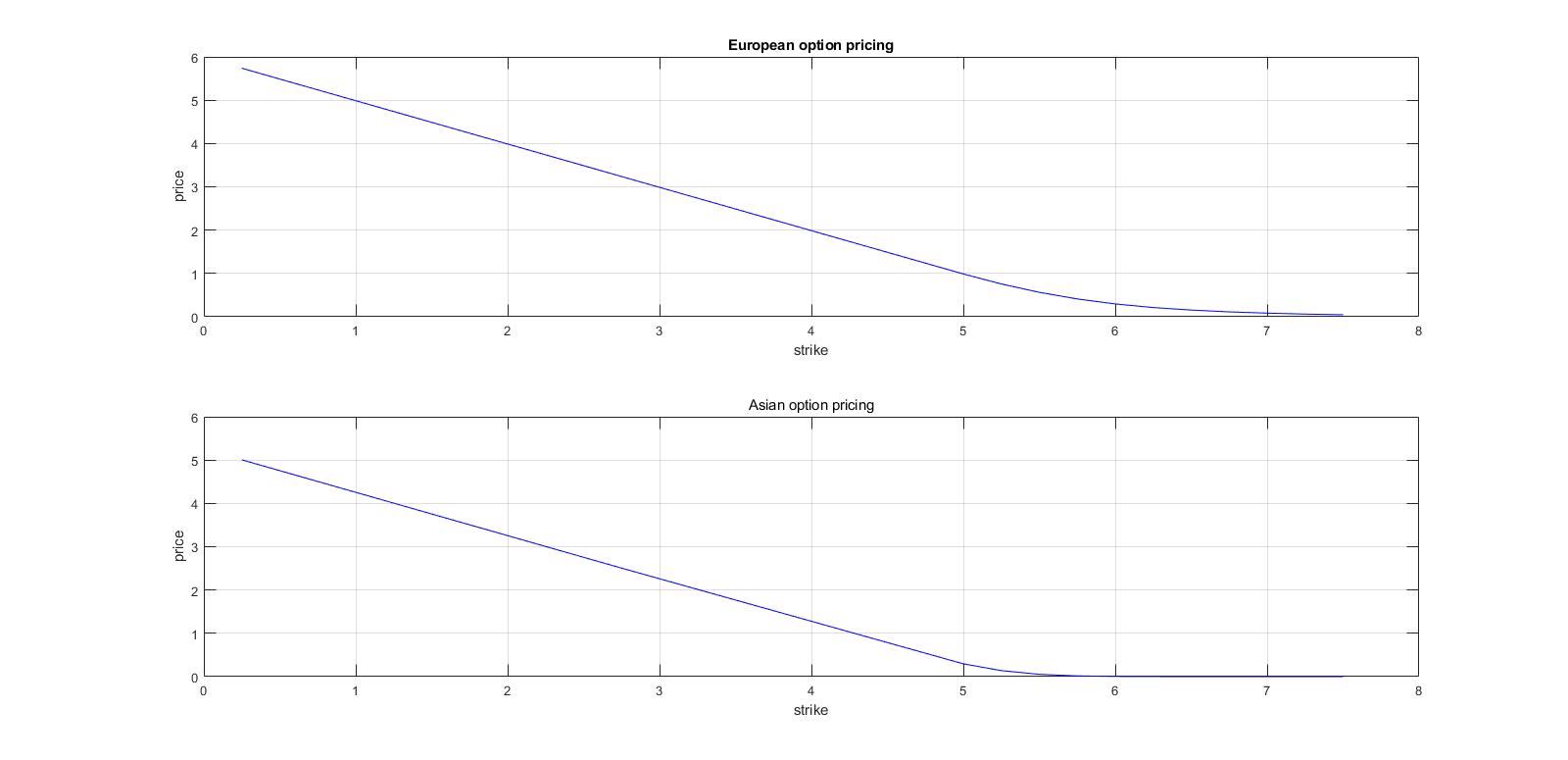}
	\begin{center}\caption{Pricing Asian and European call options.}\label{figure1}
	\end{center}
\end{figure}
To calculate the delta, we see the exact delta expression for European call option and simple Asian call option, respectively, as follow.
\begin{equation*}
	\Delta_0=\mathbb{E}[H_K(S_T)\frac{S_T}{S_0}],
\qquad \text{and}\qquad 
	\Delta_0=\mathbb{E}[H_K(Y_T)\frac{Y_T}{S_0}].
\end{equation*}
The method WM (Wiener-Malliavin weight) is based on \eqref{dwweight} and \eqref{dwweightAsian} for European and Asian options, respectively, and the method PM (Poisson-Malliavin weight) is based on \eqref{dNweight} and \eqref{dNweightAsian} for European and Asian options, respectively. The method WP (Wiener-Poisson weight) is based on the combination of two methods WM and PM. Whereas the symmetric finite difference approach for European call option gives
\begin{equation*}
	\Delta=\frac{\partial}{\partial S_0}\mathbb{E}[\max(S_T-K,0)]=\frac{\mathbb{E}[f(S_T)\vert S_0+h]-\mathbb{E}[f(S_T)\vert S_0-h]}{2h},
\end{equation*}
the symmetric finite difference approach for simple Asian call option gives
\begin{equation*}
	\Delta=\frac{\partial}{\partial S_0}\mathbb{E}[\max(Y_T-K,0)]=\frac{\mathbb{E}[f(Y_T)\vert S_0+h]-\mathbb{E}[f(Y_T)\vert S_0-h]}{2h},
\end{equation*}
The execution time of the program code in the Malliavin method WP and finite difference method in the European call option are $0.6345\times10^{5}$ and $1.2530\times10^{5}$,respectively. in the simple Asian call option are $1.0688\times10^{5}$ and $1.2621\times10^{5}$.\\
Figures \ref{figure4} and \ref{figure3} show the behavior of the terms $\Delta$, $\Delta^W$, $\Delta^N$ and $\Delta^W/2+\Delta^N/2$. The jumps are generated by Hawkes processes.
\begin{figure}[H]
	\centerline{
		\includegraphics[scale=0.27]{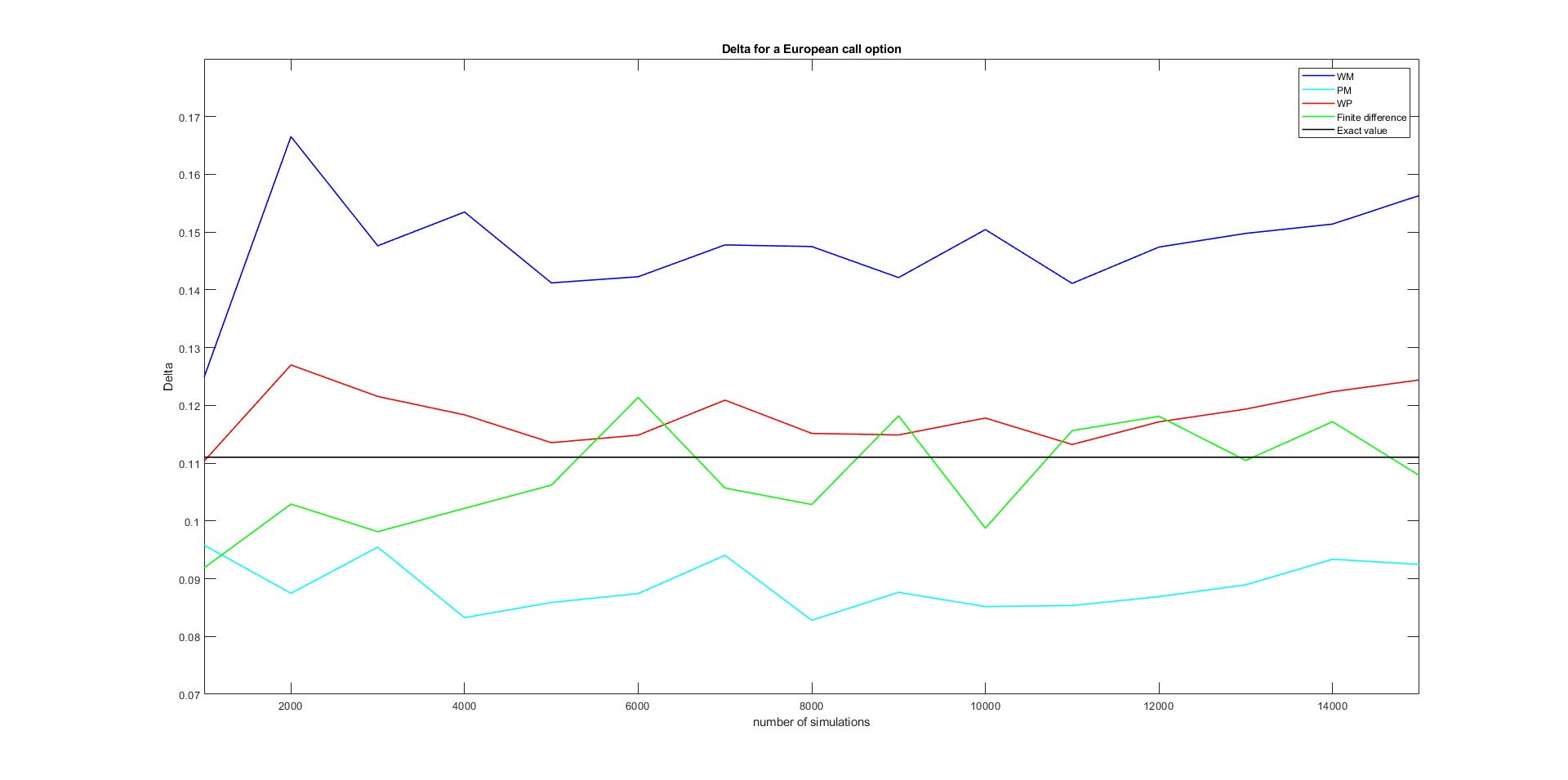}
	}
	\caption{The delta of the European call option under Hawkes process dynamics.}\label{figure4}
\end{figure}
\begin{figure}[H]
	\centerline{\includegraphics[scale=0.27]{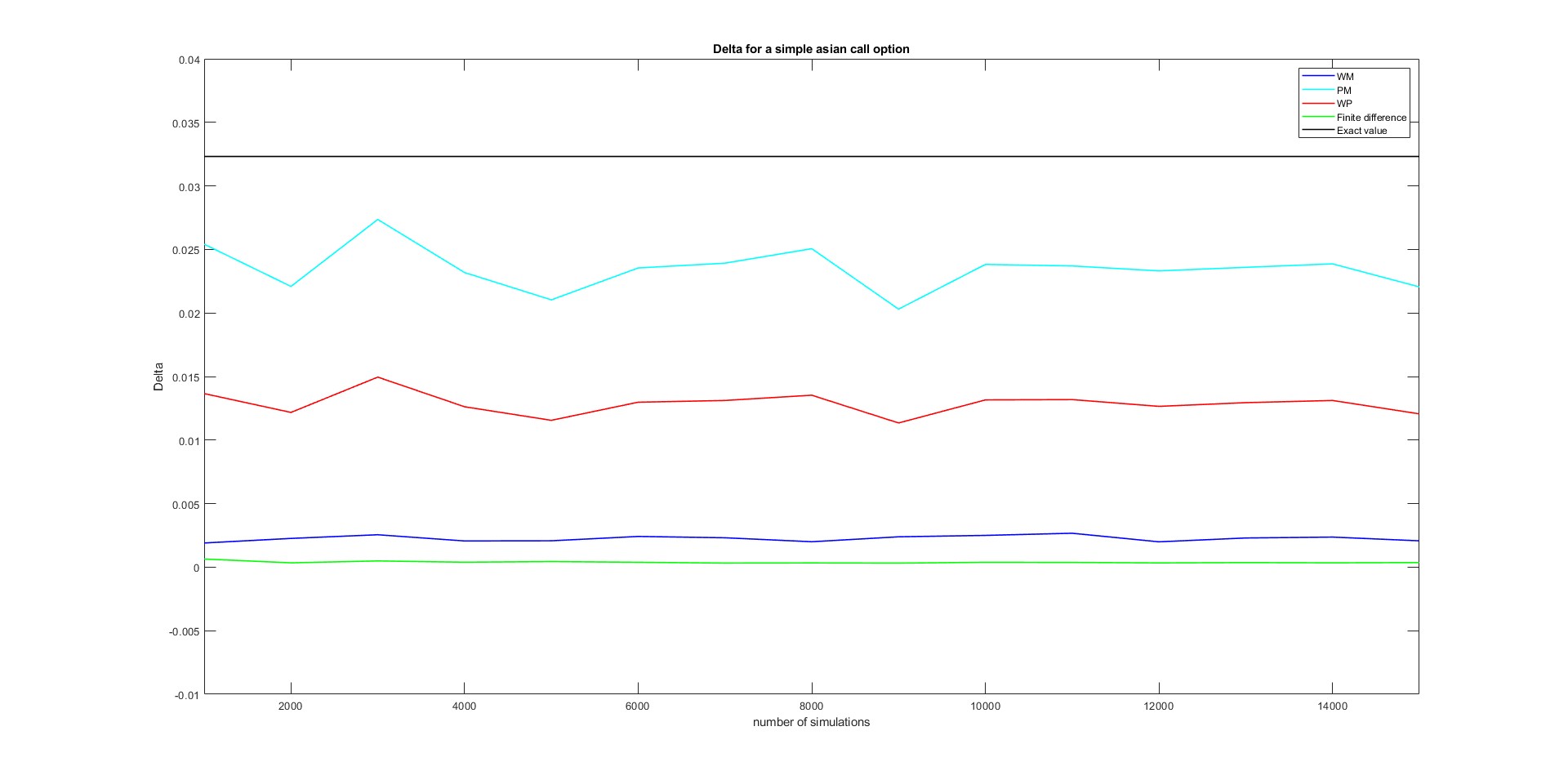}
	}
	\caption{The delta of the simple Asian call option under Hawkes process dynamics.}\label{figure3}
\end{figure}
The mean squared errors for the computation of the delta of the European call option and the simple Asian call option are presented in Tables \ref{tab4} and \ref{tab5}, respectively.

\begin{table}[ht!]
	\begin{center}
		\caption{The mean squared errors of the four methods for the European call option}
		\scalebox{1}{
			\begin{tabular}{ccccc}
				\hline
				$The\ Method$ & $MSE\ of\ the\ European\ call\ option$\\
				\hline 
				$WM $ & $0.0014$\\
				$PM$ & $0.0005$\\
				$WP$ & $0.0001$\\
				$Symmetric\ Finite\ Difference$ &  $0.0001$\\
				\hline
		\end{tabular}}\label{tab4}
	\end{center}
\end{table}
\begin{table}[ht!]
	\begin{center}
		\caption{The mean squared errors of the four methods for the simple Asian call option}
		\scalebox{1}{
			\begin{tabular}{ccccc}
				\hline
				$The\ Method$ & $MSE\ of\ the\ simple\ Asian\ option$\\
				\hline 
				$WM$ & $0.0009$\\
				$PM$ & $0.0001$\\
				$WP$ & $0.0004$\\
				$Symmetric\ Finite\ Difference$ &  $0.0010$\\
				\hline
		\end{tabular}}\label{tab5}
	\end{center}
\end{table}
\section{Conclusions}
The main purpose of this article is to compute the delta of exotic options, like Asian options, on the underlying assets drived as Hawkes jump-diffusion models, by Malliavin calculus. This research based on the relationship between the integral operators with respect to the Hawkes process $N$ and those with respect to the associated Poisson process $\overline{N}$. This relationship helps us to control the exponential moments of the solutions of Hawkes SDEs and of the stochastic intensity process $\lambda_t$. It also leads to a new duality relation between the Skorokhod integral of a process with respect to $N$ and the Malliavin derivative of that with respect to $\overline{N}$. As an application, we establish a new Wiener–Poisson weight to compute the delta for both European and Asian options, and compare the results with the exact solution and finite differences using numerical simulations and the Monte Carlo method. The convergence of the numerical method in computation of delta is verified in a separate section.\\


\appendix
\titleformat{\section}{\normalfont\Large\bfseries}{Appendix \thesection}{1em}{}
\titleformat{\subsection}{\normalfont\large\bfseries}{Appendix \thesection.\thesubsection}{1em}{}
\renewcommand{\subsectionautorefname}{Appendix}

\section{Proof of Lemma \ref{yvaroon}}\label{appendixA}
From It\^o formula for $f(S_t)=S_t^{-p}$ with $p \geq 1$,
\begin{align*}
	d\big(S_t^{-p}\big)
	&= S_t^{-p}\Big[ \big(-p\mu+\tfrac{p(p+1)}{2}\sigma^2\big)\,dt - p\sigma\,dW_t\Big]
	+ \int_{\mathbb{R}_0} S_{t^-}^{-p}\!\left[\left(1+(e^{J_t}-1)1_{(0,\lambda(t)]}(z)\right)^{-p}-1\right]\tilde{\overline{N}}(dt\times dz) \\
	& + \int_{\mathbb{R}_0} S_{t^-}^{-p}\!\left[\left(1+(e^{J_t}-1)1_{(0,\lambda(t)]}(z)\right)^{-p}-1 + p(e^{J_t}-1)1_{(0,\lambda(t)]}(z)\right]\nu(dz)\,dt.
\end{align*}
or
\begin{align*}
	d\big(S_t^{-p}\big)
	&= S_t^{-p}\Big[ \big(-p\mu+\tfrac{p(p+1)}{2}\sigma^2\big)\,dt - p\sigma\,dW_t\Big] +\int_{\mathbb{R}_0} S_{t^-}^{-p}(e^{-pJ_t}-1)1_{(0,\lambda(t)]}(z)\tilde{\overline{N}}(dt\times dz)\\
	& + \int_{\mathbb{R}_0} S_{t^-}^{-p}\big(e^{-pJ_t}+pe^{J_t}-(1+p)\big) 1_{(0,\lambda(t)]}(z)dzdt.
\end{align*}
Now $S_t^{-p} =S_0^{-p} \exp\!\left\{ -p X_t \right\}$, where
\begin{equation*}
X_t = \Big(\mu- \tfrac{1}{2}\sigma^{2}\Big)t + \sigma W_t 
      + \int_0^t \big(J_s - e^{J_s} + 1\big)\lambda(s)ds
      + \int_0^t \int_{\mathbb{R}_0} J_s 1_{(0,\lambda(t)]}(z)\tilde{\overline{N}}(ds,dz).
\end{equation*}

Expanding explicitly,
\begin{align*}
S_t^{-p}
&= S_0^{-p}\exp\Bigg\{
   -p(\mu- \tfrac{1}{2}\sigma^{2})t
   -p\sigma W_t \\
&-p\int_0^t (J_s - e^{J_s} + 1)\lambda(s)\,ds
   -p\int_0^t \int_{\mathbb{R}_0} J_s 1_{(0,\lambda(t)]}(z)\tilde{\overline{N}}(ds,dz).
\Bigg\}.
\end{align*}
or
\begin{align*}
	S_t^{-p}
	&= S_0^{-p}\exp\Bigg\{
	-p(\mu- \tfrac{1}{2}\sigma^{2})t
	-p\sigma W_t \\
	&+p\int_0^t (e^{J_s} - 1)\lambda(s)\,ds
	-p\int_0^t \int_{\mathbb{R}_0} J_s 1_{(0,\lambda(t)]}(z){\overline{N}}(ds,dz)
	\Bigg\}.
\end{align*}
Using {\bf Condition H1}, and Lemma \ref{ppp0} we conclude
\begin{align*}
\mathbb{E}(S_t^{-p})&\leq\mathbb{E}\Big(
S_0^{-p}\exp\Big\{
-p(\mu- \tfrac{1}{2}\sigma^{2})t
-p\sigma W_t 
+p\int_0^t (e^{J_s} - 1)\lambda(s)ds
\Big\}\Big)\\
&\leq\frac{1}{2}S_0^{-2p}\Big[\mathbb{E}\Big(
\exp\Big\{
-2p\big((\mu- \tfrac{1}{2}\sigma^{2})t
+\sigma W_t \big)\Big\}\Big)+\mathbb{E}\Big(
\exp\Big\{ 2p\sup_{0 \leq s \leq t}\lambda(s)\int_0^T (e^{J_s} - 1)ds\Big\}\Big)\Big]<\infty.
\end{align*}
\section{Proof of Lemma \ref{deltadomain}}\label{appendixB}
According to Theorem 13.15 in \cite{nunno}, it is sufficient to show that $1_{(0, \lambda(t)]}(z)\mathcal{H}_1$ and $1_{(0, \lambda(t)]}(z)\mathcal{H}_2$ belong to $ \mathbb{D}_N^{1,2}$. To do this, we first note that $1+\frac{K}{S_T} <2$, if $S_T > K$, and therefore
\begin{align*}
\int_0^T \int_{\mathbb{R}_+}& \E(\vert 1_{(0, \lambda(t)]}(z)\mathcal{H}_1(t,z) \vert^2)dz dt \\
& \leq 2\int_0^T \int_{\mathbb{R}_+}\E( 1_{(0, \lambda(t)]}(z)\frac{1}{\mathcal{B}_1 (e^{D_{t,z}^{\overline{N}}X_T}-1)^2})dz dt\\
& \leq 2\int_0^T \int_{\mathbb{R}_+}\E( 1_{(0, \lambda(t)]}(z)\frac{1}{(D_{t,z}^{\overline{N}}X_T)^4})dz dt
+2\int_0^T \int_{\mathbb{R}_+}\E( 1_{(0, \lambda(t)]}(z)\frac{1}{\mathcal{B}_1 ^4}) dz dt\\
& \leq 2\int_0^T \int_{\mathbb{R}_+}\E( 1_{(0, \lambda(t)]}(z)\frac{1}{(D_{t,z}^{\overline{N}}X_T)^4})dz dt
+2\int_0^T \int_{\mathbb{R}_+}\E( 1_{(0, \lambda(t)]}(z)\int_{B}\frac{1}{\lambda(s) ^4}ds) dz dt\\
&\leq 2\int_0^T \int_{\mathbb{R}_+}\E( 1_{(0, \lambda(t)]}(z)\frac{1}{(D_{t,z}^{\overline{N}}X_T)^4})dz dt
+2\int_0^T \int_{\mathbb{R}_+}\frac{1}{\lambda_0 ^4}\E( 1_{(0, \lambda(t)]}(z)) dz dt <\infty,
\end{align*}
where we used Lemma \ref{varoon} and \eqref{lamb1} in the last inequality. In a similar way, 
\begin{align*}
\int_0^T \int_{\mathbb{R}_+}\E(\vert 1_{(0, \lambda(t)]}(z)\mathcal{H}_2(t,z) \vert^2)dz dt 
&\leq \int_0^T \int_{\mathbb{R}_+}\E( 1_{(0, \lambda(t)]}(z)\frac{1}{\mathcal{B}_2^2}) dz dt\\
& \leq \int_0^T \int_{\mathbb{R}_+}\E( 1_{(0, \lambda(t)]}(z)\int_{B}\frac{1}{\lambda(s) ^2}ds) dz dt\\
&\leq \int_0^T \int_{\mathbb{R}_+}\frac{1}{\lambda_0 ^2}\E( 1_{(0, \lambda(t)]}(z)) dz dt <\infty.
\end{align*}
Then, $1_{(0, \lambda(t)]}(z)1_{B \cap ([v_1, T] \times  \mathbb{R}_0)}\mathcal{H}_1$ and $1_{(0, \lambda(t)]}(z)1_{B^c \cap ([0, v_1] \times  \mathbb{R}_0)}\mathcal{H}_2$ belong  $L^2(\Omega \times [0, T] \times \mathbb{R}_0)$.\\
In the sequence, we see that for every $w \in \mathcal{A}_1:= \{w \in \Omega, 1_{(0, \lambda(t)]}(z)1_{B \cap ([v_1, T] \times  \mathbb{R}_0)}(t,z)=1\}$, the variable $\mathcal{H}_1(t,z)$, and for every $\mathcal{A}_2:= \{w \in \Omega, 1_{(0, \lambda(t)]}(z)1_{B^c \cap ([0, v_1] \times  \mathbb{R}_0)}(t,z)=1\}$, the variable $\mathcal{H}_2(t,z)$ are differentiale whose derivatives belong $L^2(\Omega \times [0, T] \times \mathbb{R}_0)$ in two steps. \\
To this end, we apply the following lemma from \cite{n19}.
\begin{lemma}\cite{n19}\label{lemma10}
Given $F,G\in\mathbb{D}_{\overline{N}}^{1,2}$ that $FG\in L^{2}(\Omega)$ and $(F + D^{\overline{N}}F)(G + D^{\overline{N}}G)\in L^{2}(\Omega\times [0,T] \times \mathbb{R}_0)$, the product $FG$ also belongs to $\mathbb{D}_{\overline{N}}^{1,2}$ and
	\begin{equation*}
	\nonumber D_{t,z}^{\overline{N}}(FG)=FD_{t,z}^{\overline{N}}G+GD_{t,z}^{\overline{N}}F+D_{t,z}^{\overline{N}}FD_{t,z}^{\overline{N}}G.
	\end{equation*}
\end{lemma}
As a consequence, we set $F=\frac{F_1}{G_1}$ and $G=G_1$ for every $F_1, G_1 \in \mathbb{D}_{\overline{N}}^{1, 2}$,  such that $\frac{F_1}{G_1} \in \mathbb{D}_{\overline{N}}^{1, 2}$ and the right hand side of the following equality belongs $ L^{2}(\Omega\times [0,T] \times \mathbb{R}_0)$, we deduce 
\begin{align*}
 D^{\overline{N}}\left(\frac{F_1}{G_1}\right)=\frac{G_1 D^{\overline{N}} F_1-F_1 D^{\overline{N}} G_1}{G_1(G_1+D^{\overline{N}} G_1)}.
\end{align*}
\begin{enumerate}
\item 
First we show that $D^{\overline{N}}(\mathcal{H}_1) \in L^2(\Omega \times [0, T] \times \mathbb{R}_0)$.
We are moving forward step by step to do this. The Malliavin derivative of the process $e^{D^{\overline{N}}_{t, z} X_T}$ is derived in \autoref{appendixD}, explicitly for the rest of the reader, nevertheless; this explicit representation is not directly needed in this proof.
\begin{enumerate}
	\item 
For every $0 \le r \leq t$ and $ z_0 \leq \lambda(r)$, 
	\begin{align*}
		\left \vert D^{\overline{N}}_{r, z_0}\left[\frac{1}{e^{D^{\bar{N}}_{t, z} X_T}-1}\right]\right\vert &= \left\vert \frac{e^{D^{\overline{N}}_{t, z} X_T}\left[e^{D^{\overline{N}}_{r, z_0} D^{\overline{N}}_{t, z} X_T}-1\right]}{\left(e^{D^{\overline{N}}_{t, z} X_T}-1\right)\left[e^{D^{\overline{N}}_{r, z_0} D^{\bar{N}}_{t, z} X_T+D^{\overline{N}}_{t, z} X_T}-1\right]}\right\vert\\
		 &\leq \left\vert\frac{e^{D^{\overline{N}}_{t, z} X_T}}{e^{D^{\overline{N}}_{t, z} X_T}-1}\right\vert \leq 1+ \left\vert\frac{1}{D^{\overline{N}}_{t, z} X_T}\right\vert ,
	\end{align*}
	where we used the facts $D^{\overline{N}}_{t, z} X_T>0$ 
	in the last inequality. Then using Lemmas \ref{ppp0} and \ref{varoon}
	\begin{align}\label{DINT}
		\nonumber
		\mathbb{E}\Big(1_{\mathcal{A}_1}\int_0^T \int_{\mathbb{R}_+} &1_{\{z_0 \leq \lambda(r)\}} \left\lvert D^{\overline{N}}_{r, z_0}\left[\frac{1}{e^{D^{\overline{N}}_{t, z} X_T}-1}\right]\right\vert^2 d z_0 d r\Big) \\ 
		\nonumber
		& \leq  \mathbb{E}\left(\int_0^T \int_{\mathbb{R}_+}1_{\{z_0 \leq\lambda(r)\}} dz_0 dr\right) + \mathbb{E}\left((\frac{1_{(0, \lambda(t)]}(z)}{ D^{\overline{N}}_{t, z} X_T})^2\int_0^T \int_{\mathbb{R}_+}1_{\{z_0 \leq \lambda(r)\}}  dz_0 d r\right)\\
		& \leq T \sup _{0 \leq r \leq T} \mathbb{E}(\lambda(r))+ \mathbb{E}^{\frac{1}{2}}\left((\frac{1_{(0, \lambda(t)]}(z)}{ D^{\overline{N}}_{t, z} X_T})^4\right)\mathbb{E}^{\frac{1}{2}}\left(\int_0^T \int_{\mathbb{R}_+}1_{\{z_0 \leq \lambda(r)\}}  dz_0 d r\right)
		<\infty .
	\end{align}
In addition, from Lemma \ref{lemma10} it follows that
	\begin{equation*}
		D_{r,z_0}(\frac{K}{S_T})=\frac{-KD_{r,z_0}S_T}{S_T(S_T+D_{r,z\_0}S_T)}=\frac{KS_T(e^{D_{r,z_0}X_T}-1)}{S_T^{2}e^{D_{r,z_0}X_T}}=\frac{K(e^{D_{r,z_0}X_T}-1)}{S_Te^{D_{r,z_0}X_T}}.
	\end{equation*}
	Therefore, given $S_T> K$ and Lemma \ref{yvaroon}, we have that $\lvert\frac{K}{S_T}\rvert< 1$
	and also $\frac{e^{D_{r,z_0}X_T}-1}{e^{D_{r,z_0}X_T}}<1$, so it can be concluded that
	\begin{equation}\label{DKST}
		\left\lvert D_{r,z_0}(\frac{K}{S_T})\right\rvert^{2}< 1.
	\end{equation}
	On the other hand, since $D_{r,z_0}(\frac{1}{\mathcal{B}_1})=\frac{-D_{r,z_0}\mathcal{B}_1}{\mathcal{B}_1(\mathcal{B}_1+D_{r,z_0}\mathcal{B}_1)}$, therefore
	\begin{equation}\label{D1B}
		\mathbb{E}\Big(1_{\mathcal{A}_1}\int_0^T \int_{\mathbb{R}_+} 1_{\{z_0 \leq \lambda(r)\}}\left\lvert D_{r,z_0}(\frac{1}{\mathcal{B}_1})\right\rvert^{2}dz_0dr\Big)\leq \mathbb{E}\Big(\frac{\int_0^T \lambda(r)dr}{\mathcal{B}_1^4}\Big)\leq \mathbb{E}\Big(\frac{1}{\int_{0}^{T}\lambda_{s}ds}\Big)^{2}<\infty.
	\end{equation}
	\item
	We now prove that $D_{r,z_0}^{\overline{N}}\big((1+\frac{K}{S_T})H_K(S_T)\big)$ and $D^{\overline{N}}_{r, z_0}\left[\frac{1}{\mathcal{B}_1\big(e^{D^{\overline{N}}_{t, z}X_T}-1\big)}\right]$ belong to $L^2(\Omega \times [0, T] \times \mathbb{R}_0)$. From Lemma \ref{lemma10} and equation \eqref{DKST}, we result
	\begin{align*}
	\mathbb{E}\Big(1_{\mathcal{A}_1}\int_0^T &\int_{\mathbb{R}_+} 1_{\{z_0 \leq \lambda(r)\}}\lvert D_{r,z_0}^{\overline{N}}\big((1+\frac{K}{S_T})H_K(S_T)\big)\rvert^{2}dz_0dr\Big )\\
	&\leq \mathbb{E}\Big(1_{\mathcal{A}_1}\int_0^T \int_{\mathbb{R}_+} 1_{\{z_0 \leq \lambda(r)\}}\lvert H_K(S_T)D_{r,z_0}^{\overline{N}}\big(1+\frac{K}{S_T}\big)\rvert^{2}dz_0dr\Big)\\
	&+\mathbb{E}\Big(1_{\mathcal{A}_1}\int_0^T \int_{\mathbb{R}_+} 1_{\{z_0 \leq \lambda(r)\}}\lvert(1+\frac{K}{S_T})	D_{r,z_0}^{\overline{N}}\big(H_K(S_T)\big)\rvert^{2}dz_0dr\Big)\\
	&+\mathbb{E}\Big(1_{\mathcal{A}_1}\int_0^T \int_{\mathbb{R}_+} 1_{\{z_0 \leq \lambda(r)\}}\lvert D_{r,z_0}^{\overline{N}}\big(H_K(S_T)\big)D_{r,z_0}^{\overline{N}}\big(1+\frac{K}{S_T}\big)\rvert^{2}dz_0dr\Big)\\
	&<\mathbb{E}\Big(1_{\mathcal{A}_1}\int_0^T \lambda(r)dr\Big)(1+4+1)<\infty,
	\end{align*}
	and	from \eqref{DINT} and \eqref{D1B}
\begin{align*}
	\mathbb{E}\Big(1_{\mathcal{A}_1}\int_0^T & \int_{\mathbb{R}_+} 1_{\{z_0 \leq \lambda(r)\}}\lvert D^{\overline{N}}_{r, z_0}[\frac{1}{\mathcal{B}_1(e^{D^{\overline{N}}_{t, z}X_T}-1)}]\rvert^{2}dz_0dr\Big)\\
	&\leq 2\mathbb{E}\Big(1_{\mathcal{A}_1}\int_0^T \int_{\mathbb{R}_+} 1_{\{z_0 \leq \lambda(r)\}}\lvert D_{r,z_0}^{\overline{N}}(\frac{1}{\mathcal{B}_1})\frac{1}{e^{D^{\overline{N}}_{t, z}X_T}-1}\rvert^{2}dz_0dr\Big)\\
	&+
	2\mathbb{E}\Big(1_{\mathcal{A}_1}\int_0^T \int_{\mathbb{R}_+} 1_{\{z_0 \leq \lambda(r)\}}\lvert \frac{1}{\mathcal{B}_1}D_{r,z_0}^{\overline{N}}(\frac{1}{e^{D^{\overline{N}}_{t, z}X_T}-1})\rvert^{2}dz_0dr\Big)\\
    &+2\mathbb{E}\Big(1_{\mathcal{A}_1}\int_0^T \int_{\mathbb{R}_+} 1_{\{z_0 \leq \lambda(r)\}}\lvert D_{r,z_0}^{\overline{N}}(\frac{1}{\mathcal{B}_1})D_{r,z_0}^{\overline{N}}(\frac{1}{e^{D^{\overline{N}}_{t, z}X_T}-1})\rvert^{2}dz_0dr\Big)<\infty.
\end{align*}	
\item From Lemma \ref{lemma10} and the parts (a) and (b) we conclude $ D^{\overline{N}}_{r, z_0}\mathcal{H}_1(t,z)\in L^2(\Omega \times [0, T] \times \mathbb{R}_0)$.
\end{enumerate}
\item In a similar way, one can easily show that $D^{\overline{N}}_{r, z_0}\mathcal{H}_2(t,z)\in {L^{2}(\Omega\times [0,T]\times \mathbb{R}_{0})}$.   
\end{enumerate}
\section{Proof of Lemma \ref{deltadomainY}}\label{appendixC}
 We see that for every $w \in \mathcal{J}_1:= \{w \in \Omega, 1_{(0, \lambda(t)]}(z)1_{C_Y \cap ([v_1, T] \times  \mathbb{R}_0)}(t,z)=1\}$, the variable $\mathcal{H}_3(t,z)$, and for every $\mathcal{J}_2:= \{w \in \Omega, 1_{(0, \lambda(t)]}(z)1_{C_Y^c \cap ([0, v_1] \times  \mathbb{R}_0)}(t,z)=1\}$, the variable $\mathcal{H}_4(t,z)$ are differentiable whose derivatives belong $L^2(\Omega \times [0, T] \times \mathbb{R}_0)$ in two steps.\\
\begin{enumerate}
	\item 
	First we show that $D^{\overline{N}}(\mathcal{H}_3) \in L^2(\Omega \times [0, T] \times \mathbb{R}_0)$.
	\begin{enumerate}
		\item Using Jensen inequality for the function $x\rightarrow\frac{1}{x^{2}}$,
\begin{align}\label{EINTT}
\nonumber	\mathbb{E}\Big(\frac{1}{\int_{t}^{T}S_u(e^{D^{\overline{N}}_{t,z}X_u}-1)du}\Big)^{2}&\leq\mathbb{E}\Big(\int_{t}^{T}\frac{1}{S_u^{2}(e^{D^{\overline{N}}_{t,z}X_u}-1)^{2}}\Big)\\
	&\leq T^{2}\mathbb{E}^{\frac{1}{2}}(\frac{1}{S_u^{4}})\mathbb{E}^{\frac{1}{2}}(\frac{1}{e^{D^{\overline{N}}_{t,z}X_u}-1})^{4}<\infty.
\end{align}
For every $0 \le r \leq t$ and $ z_0 \leq \lambda(r)$,
\begin{align*}
&\left\vert D^{\overline{N}}_{r,z_0}\Big(\frac{1}{\int_{t}^{T}S_u(e^{D^{\overline{N}}_{t,z}X_u}-1)du}\Big)\right\vert\\	&=\left\vert\frac{D^{\overline{N}}_{r,z_0}\int_{t}^{T}S_u(e^{D^{\overline{N}}_{t,z}X_u}-1)du}{\int_{t}^{T}S_u(e^{D^{\overline{N}}_{t,z}X_u}-1)du\big(\int_{t}^{T}S_u(e^{D^{\overline{N}}_{t,z}X_u}-1)du+D^{\overline{N}}_{r,z_o}\int_{t}^{T}S_u(e^{D^{\overline{N}}_{t,z}X_u}-1)du\big)} \right\vert\\
&=\left\vert \frac{\int_{r}^{T}D^{\overline{N}}_{r,z_0}\big(S_u(e^{D^{\overline{N}}_{t,z}X_u}-1)du\big)}{\int_{t}^{T}S_u(e^{D^{\overline{N}}_{t,z}X_u}-1)du\Big(\int_{t}^{T}S_u(e^{D^{\overline{N}}_{t,z}X_u}-1)du+\int_{r}^{T}\big(D^{\overline{N}}_{r,z_o}S_u(e^{D^{\overline{N}}_{t,z}X_u}-1)\big)du\Big)} \right\vert\\
&\leq
\left\vert \frac{\int_{r}^{T}D^{\overline{N}}_{r,z_0}\big(S_u(e^{D^{\overline{N}}_{t,z}X_u}-1)du\big)}{\Big(\int_{t}^{T}S_u(e^{D^{\overline{N}}_{t,z}X_u}-1)du\Big)^{2}+\int_{t}^{T}S_u(e^{D^{\overline{N}}_{t,z}X_u}-1)du\int_{r}^{T}\big(D^{\overline{N}}_{r,z_o}S_u(e^{D^{\overline{N}}_{t,z}X_u}-1)\big)du} \right\vert\\
&\leq
\left\vert \frac{1}{\int_{t}^{T}S_u(e^{D^{\overline{N}}_{t,z}X_u}-1)du} \right\vert.
\end{align*} 
Using equation \eqref{EINTT}, we have
\begin{align*}\label{DINTT}
	\nonumber
	&\mathbb{E}\Big(1_{\mathcal{J}_1}\int_0^T \int_{\mathbb{R}_+} 1_{\{z_0 \leq \lambda(r)\}} \left\lvert D^{\overline{N}}_{r, z_0}\left[\frac{1}{\int_{t}^{T}S_u(e^{D^{\overline{N}}_{t,z}X_u}-1)du}\right]\right\vert^2 d z_0 d r\Big) \\ 
	& \leq   \mathbb{E}\left(\int_0^T \int_{\mathbb{R}_+}1_{\{z_0 \leq \lambda(r)\}} \left(\frac{1_{(0,\lambda(t)]}(z)}{ \int_{t}^{T}S_u(e^{D^{\overline{N}}_{t,z}X_u}-1)du}\right)^2 dz_0 dr\right)<\infty.
\end{align*}
Since $D^{\overline{N}}_{r,z_0}(\frac{1}{\mathcal{C}_1})=\frac{-D^{\overline{N}}_{r,z_0}\mathcal{C}_1}{\mathcal{C}_1(\mathcal{C}_1+D^{\overline{N}}_{r,z_0}\mathcal{C}_1)}$, therefore
\begin{equation}\label{D1C}
	\mathbb{E}\Big(1_{\mathcal{J}_1}\int_0^T \int_{\mathbb{R}_+} 1_{\{z_0 \leq \lambda(r)\}}\left\lvert D^{\overline{N}}_{r,z_0}(\frac{1}{\mathcal{C}_1})\right\rvert^{2} _0 dz_0dr\Big)\leq \mathbb{E}\Big(\frac{1}{\mathcal{C}_1^4}(\int_0^T \lambda(r) dr)^2\Big)\leq \mathbb{E}\Big(\frac{1}{\int_{0}^{T}\lambda(s)ds}\Big)^{2}<\infty.
\end{equation}
According to Hölder inequality, Lemmas \ref{lemma4.1} and \ref{dlambdaexp}, connection with the proof of Theorem 3.1 in \cite{n20} (on page 2116), we conclude
\begin{align}\label{DKYT}
\nonumber	
\mathbb{E} \Big(&1_{\mathcal{J}_1}\int_0^T \int_{\mathbb{R}_+} 1_{\{z_0 \leq \lambda(r)\}}\lvert D^{\overline{N}}_{r,z_0}(K+Y_T)\rvert^{2}dz_0dr\Big)\notag \\
\nonumber
&=\frac{1}{T}\mathbb{E}\Big( 1_{\mathcal{J}_1}\int_0^T \int_{\mathbb{R}_+} 1_{\{z_0 \leq \lambda(r)\}}\lvert\int_{r}^{T}S_u(e^{D^{\overline{N}}_{r,z_0}X_u}-1)du\rvert^{2}dz_0dr\Big)\\
\nonumber
&\leq
\frac{1}{2}\mathbb{E}\Big( 1_{\mathcal{J}_1}\int_0^T \int_{\mathbb{R}_+} 1_{\{z_0 \leq \lambda(r)\}}\Big[\int_{r}^{T}S_u^2du+ \int_{r}^{T}(e^{D^{\overline{N}}_{r,z_0}X_u}-1)^2 du\Big]dz_0dr\Big)\\
\nonumber
&\leq
 \frac{1}{2}\mathbb{E}\Big( \int_0^T \int_{\mathbb{R}_+} \lambda(r)\Big[\int_{r}^{T}S_u^2du+ \int_{r}^{T}(e^{D^{\overline{N}}_{r,z_0}X_u}-1)^2 du\Big]dz_0dr\Big)<\infty.\\
 \end{align}
\item
We now prove that $D_{r,z_0}^{\overline{N}}\big((S_T+Y_T)H_K(S_T)\big)$ and $D^{\overline{N}}_{r, z_0}\left[\frac{1}{\mathcal{C}_1\int_{t}^{T}S_u(e^{D^{\overline{N}}_{t,z}X_u}-1)du}\right]$ belong to $L^2(\Omega \times [0, T] \times \mathbb{R}_0)$. From Lemma \ref{lemma10} and equation \eqref{DKYT}, we can write
	\begin{align*}
	\mathbb{E}\Big(1_{\mathcal{J}_1}\int_0^T &\int_{\mathbb{R}_+} 1_{\{z_0 \leq \lambda(r)\}}\lvert D_{r,z_0}^{\overline{N}}\big((K+Y_T)H_K(Y_T)\big)\rvert^{2}dz_0dr\Big )\\
	&\leq \mathbb{E}\Big(1_{\mathcal{J}_1}\int_0^T \int_{\mathbb{R}_+} 1_{\{z_0 \leq \lambda(r)\}}\lvert H_K(Y_T)D_{r,z_0}^{\overline{N}}\big(K+Y_T\big)\rvert^{2}dz_0dr\Big)\\
	&+\mathbb{E}\Big(1_{\mathcal{J}_1}\int_0^T \int_{\mathbb{R}_+} 1_{\{z_0 \leq \lambda(r)\}}\lvert(K+Y_T)	D_{r,z_0}^{\overline{N}}\big(H_K(Y_T)\big)\rvert^{2}dz_0dr\Big)\\
	&+\mathbb{E}\Big(1_{\mathcal{J}_1}\int_0^T \int_{\mathbb{R}_+} 1_{\{z_0 \leq \lambda(r)\}}\lvert D_{r,z_0}^{\overline{N}}\big(H_K(Y_T)\big)D_{r,z_0}^{\overline{N}}\big(K+Y_T\big)\rvert^{2}dz_0dr\Big)<\infty,
\end{align*}
From \eqref{DINT} and \eqref{D1C}
\begin{align*}
	\mathbb{E}\Big(1_{\mathcal{J}_1}\int_0^T & \int_{\mathbb{R}_+} 1_{\{z_0 \leq \lambda(r)\}}\lvert D^{\overline{N}}_{r, z_0}[\frac{1}{\mathcal{C}_1\int_{t}^{T}S_u(e^{D^{\overline{N}}_{t,z}X_u}-1)du}]\rvert^{2}dz_0dr\Big)\\
	&\leq 2\mathbb{E}\Big(1_{\mathcal{J}_1}\int_0^T \int_{\mathbb{R}_+} 1_{\{z_0 \leq \lambda(r)\}}\lvert D_{r,z_0}^{\overline{N}}(\frac{1}{\mathcal{C}_1})\frac{1}{\int_{t}^{T}S_u(e^{D^{\overline{N}}_{t,z}X_u}-1)du}\rvert^{2}dz_0dr\Big)\\
	&+
	2\mathbb{E}\Big(1_{\mathcal{J}_1}\int_0^T \int_{\mathbb{R}_+} 1_{\{z_0 \leq \lambda(r)\}}\lvert \frac{1}{\mathcal{C}_1}D_{r,z_0}^{\overline{N}}(\frac{1}{\int_{t}^{T}S_u(e^{D^{\overline{N}}_{t,z}X_u}-1)du})\rvert^{2}dz_0dr\Big)\\
	&+2\mathbb{E}\Big(1_{\mathcal{J}_1}\int_0^T \int_{\mathbb{R}_+} 1_{\{z_0 \leq \lambda(r)\}}\lvert D_{r,z_0}^{\overline{N}}(\frac{1}{\mathcal{C}_1})D_{r,z_0}^{\overline{N}}(\frac{1}{\int_{t}^{T}S_u(e^{D^{\overline{N}}_{t,z}X_u}-1)du})\rvert^{2}dz_0dr\Big)<\infty.
\end{align*}
\item From Lemma \ref{lemma10} and the parts (a) and (b) we conclude $ D^{\overline{N}}_{r, z_0}\mathcal{H}_3(t,z)\in L^2(\Omega \times [0, T] \times \mathbb{R}_0)$.
\end{enumerate}
\item In a similar way, one can easily show that $D^{\overline{N}}_{r, z_0}\mathcal{H}_4(t,z)\in {L^{2}(\Omega\times [0,T]\times \mathbb{R}_{0})}$.
\end{enumerate}
\section{}\label{appendixD}
To present an expression for the derivative of $e^{D^{\overline{N}}_{t, z} X_T}$, we are motivated from \cite{n20} and derive the following results.  
\begin{align*}
 & D^{\overline{N}}_{r, z_0}\Big(\int_u^t J_s N_{(u,z)}(ds) \Big)\\
 & =\int_u^t \int_{\mathbb{R}_+}1_{s \leq r} J_s\left(1_{\Big(0, \varphi_s(\overline{N}+\varepsilon_{(u, z)}+\varepsilon_{(r, z_0)}|_{(-\infty, s)})\Big]}(v)- 1_{\Big(0, \varphi_s(\overline{N}+\varepsilon_{(r, z_0)}|_{(-\infty, s)})\Big]}{(v)}\right) (\overline{N}+\varepsilon_{(r, z_0)} )(dv, ds)\\
&+ \int_u^t \int_{\mathbb{R}_+}1_{s > r} J_s\left(1_{\Big(0, \varphi_s(\overline{N}+\varepsilon_{(u, z)}+\varepsilon_{(r, z_0)}|_{(-\infty, s)})\Big]}{(v)}- 1_{\Big(0, \varphi_s(\overline{N}+\varepsilon_{(r, z_0)}|_{(-\infty, s)})\Big]}{(v)}\right) (\overline{N}+\varepsilon_{(r, z_0)} )(dv, ds)\\
& -\int_u^t \int_{\mathbb{R}_+}1_{s \leq r} J_s\left(1_{\Big(0, \varphi_s(\overline{N}+\varepsilon_{(u, z)}|_{(-\infty, s)})\Big]}{(v)}- 1_{\Big(0, \varphi_s(\overline{N}|_{(-\infty, s)})\Big]}{(v)}\right) \overline{N}(dv, ds)\\
& -\int_u^t \int_{\mathbb{R}_+}1_{s > r} J_s\left(1_{\Big(0, \varphi_s(\overline{N}+\varepsilon_{(u, z)}|_{(-\infty, s)})\Big]}{(v)}- 1_{\Big(0, \varphi_s(\overline{N}|_{(-\infty, s)})\Big]}{(v)}\right) \overline{N}(dv, ds)\\
& =\int_u^t \int_{\mathbb{R}_+}1_{s \leq r} J_s\Big(1_{\Big(0, \varphi_s(\overline{N}+\varepsilon_{(u, z)}|_{(-\infty, s)})\Big]}(v)- 1_{\Big(0, \varphi_s(\overline{N}|_{(-\infty, s)})\Big]}{(v)}\Big) (\overline{N}+\varepsilon_{(r, z_0)} )(dv, ds)\\
&+ \int_u^t \int_{\mathbb{R}_+}1_{s > r} J_s\left(1_{\Big(0, \varphi_s(\overline{N}+\varepsilon_{(u, z)}+\varepsilon_{(r, z_0)}|_{(-\infty, s)})\Big]}{(v)}- 1_{\Big(0, \varphi_s(\overline{N}+\varepsilon_{(r, z_0)}|_{(-\infty, s)})\Big]}{(v)}\right) \overline{N}(dv, ds)\\
& -\int_u^t \int_{\mathbb{R}_+}1_{s \leq r} J_s\left(1_{\Big(0, \varphi_s(\overline{N}+\varepsilon_{(u, z)}|_{(-\infty, s)})\Big]}{(v)}- 1_{\Big(0, \varphi_s(\overline{N}|_{(-\infty, s)})\Big]}{(v)}\right) \bar{N}(dv, ds)\\
& -\int_u^t \int_{\mathbb{R}_+}1_{s > r} J_s\left(1_{\Big(0, \varphi_s(\overline{N}+\varepsilon_{(u, z)}|_{(-\infty, s)})\Big]}{(v)}- 1_{\Big(0, \varphi_s(\overline{N}|_{(-\infty, s)})\Big]}{(v)}\right) \overline{N}(dv, ds)\\
& = 1_{r \geq u} J_u \Big(1_{\Big(0, \varphi_r(\overline{N}+\varepsilon_{(u, z)}|_{(-\infty, r)})\Big]}{(z_0)}- 1_{\Big(0, \varphi_r(\overline{N}|_{(-\infty, r)})\Big]}{(z_0)} \Big) \\
& + \int_u^t \int_{\mathbb{R}_+}1_{s > r > u} J_s A_{r,u,z,z_0}(v,s) \overline{N}(dv, ds),\\ 
 \end{align*}
where 
\begin{align*}
A_{r,u,z,z_0}(v,s)&=1_{\Big(0, \varphi_s(\overline{N}+\varepsilon_{(u, z)}+\varepsilon_{(r, z_0)}|_{(-\infty, s)})\Big]}{(v)}-1_{\Big(0, \varphi_s(\overline{N}+\varepsilon_{(r, z_0)}|_{(-\infty, s)})\Big]}{(v)}\\
&-1_{\Big(0, \varphi_s(\overline{N}+\varepsilon_{(u, z)}|_{(-\infty, s)})\Big]}{(v)}+ 1_{\Big(0, \varphi_s(\overline{N}|_{(-\infty, s)})\Big]}{(v)}.
\end{align*}
In the same manner, 
\begin{align*}
 D^{\overline{N}}_{r, z_0}D^{\overline{N}}_{(u,z)}\lambda(s) &= 1_{r \geq u} \alpha e^{-\beta(t-r)} \Big(1_{\Big(0, \varphi_r(\overline{N}+\varepsilon_{(u, z)}|_{(-\infty, r)})\Big]}{(z_0)}- 1_{\Big(0, \varphi_r(\overline{N}|_{(-\infty, r)})\Big]}{(z_0)} \Big) \\
& + \int_u^t \int_{\mathbb{R}_+}1_{s > r > u} \alpha e^{-\beta(t-s)} A_{r,u,z,z_0}(v,s) \overline{N}(dv, ds).\\ 
\end{align*}
In addition, 
\begin{align*}
D^{\overline{N}}_{r, z_0}(J_u 1_{(0, \lambda(u)]}(z))= J_u \Big(1_{(0, \lambda(u)+D^{\overline{N}}_{r, z_0} \lambda(u) ]}(z)-1_{(0, \lambda(u)}(z)\Big).
\end{align*}
\section{Proof of Lemma \ref{dlambdaexp}}\label{appendixE}
According to the equation \eqref{DNlambda}, and apply the It\^o formula for the function $f(z) = e^{pz}$, $p \geq 1$,
\begin{equation*}
	d\left( e^{p D_{u,z}^{\overline{N}} \lambda_t} \right) = - p\beta D_{u,z}^{\overline{N}} \lambda_{t} e^{p D_{u,z}^{\overline{N}} \lambda_{t}} dt + e^{p D_{u,z}^{\overline{N}} \lambda_{t}} \left( e^{p \alpha\; sign(D_{u,z}^{\overline{N}} \lambda_{t})} - 1 \right) N_{(u,z)}(dt).
\end{equation*}
Therefore, the integral form is as follows.
\begin{equation}\label{pexpmomD}
	e^{p D_{u,z}^{\overline{N}} \lambda_t} = e^{p\alpha 1_{(0,\lambda(u)]}(z)} 
	- p\beta  \int_u^t D_{u,z}^{\overline{N}} \lambda_{s} e^{p D_{u,z}^{\overline{N}} \lambda_{s}} ds
	+ \int_u^t e^{p D_{u,z}^{\overline{N}} \lambda_{s}} \left(e^{p\alpha\; sign(D_{u,z}^{\overline{N}} \lambda_{s})} - 1\right)N_{(u,z)}(ds),
\end{equation}
and from the Jensen inequality
\begin{align*}
	\E\left(e^{p\alpha 1_{(0,\lambda(u)]}(z)}\right)&= \E\left(e^{p\alpha 1_{(0,\lambda(u)]}(z)}\right) - p\beta \int_u^t  \E\left(D_{u,z}^{\overline{N}} \lambda_{s} e^{p D_{u,z}^{\overline{N}} \lambda_{s}}\right)ds\\
	& +  \int_u^t \E\left(D_{u,z}^{\overline{N}} \lambda_{s}e^{p D_{u,z}^{\overline{N}} \lambda_{s}} \big(e^{p\alpha\; sign(D_{u,z}^{\overline{N}} \lambda_{s})} - 1\big)\right)ds\\
	& \leq \E\left(e^{p\alpha 1_{(0,\lambda(u)]}(z)}\right) +\big(e^{p\alpha} - 1- p\beta \big)\int_u^t  \E\left(D_{u,z}^{\overline{N}} \lambda_{s} e^{p D_{u,z}^{\overline{N}} \lambda_{s}}\right) ds\\
	&\leq 1+(e^{p\alpha}-1)\frac{\E(\vert \lambda(u)\vert^2)}{z^2} + \big(e^{p\alpha} - 1- p\beta \big) \int_u^t  \E\left(D_{u,z}^{\overline{N}} \lambda_{s} e^{p D_{u,z}^{\overline{N}} \lambda_{s}}\right) ds.
\end{align*}	
The Gronwall inequality and then the Young inequality show that for every $p \geq 1$, 
\begin{equation}\label{suppexpD}
	\sup_{0 \leq t \leq T} \E(e^{pD_{u,z}^{\overline{N}} \lambda_t} ) < \infty.
\end{equation}
In addition, Taking expectation from the supremum of \eqref{pexpmomD}, and using $sign(D_{u,z}^{\overline{N}} \lambda_{s})>0$, then the Burkholder-Davis-Gundy inequality deduces
\begin{align*}
	\E( \sup_{0 \leq u \leq t \leq T}e^{p D_{u,z}^{\overline{N}} \lambda_t})&\leq  \max\{e^{p\alpha}, 1\}+  (e^{p\alpha} - 1) \E\Big(\Big\vert \int_u^T  e^{2 D_{u,z}^{\overline{N}} \lambda_s} D_{u,z}^{\overline{N}} \lambda_{s} ds\Big\vert^{\frac12}\Big) \\
	&+  (e^{p\alpha} - 1) \int_u^T \E( e^{pD_{u,z}^{\overline{N}} \lambda_{s}}  D_{u,z}^{\overline{N}} \lambda_{s}) ds\\
	& \leq  \max\{e^{p\alpha}, 1\} + \frac12  (e^{p\alpha} - 1)  \int_u^T  \E\Big(  e^{2D_{u,z}^{\overline{N}} \lambda_{s}}D_{u,z}^{\overline{N}} \lambda_{s}\Big) ds + \frac12 (e^{p\alpha} - 1)\\
	& + (e^{p\alpha} - 1)  \int_u^T \E\Big(\frac12 e^{2pD_{u,z}^{\overline{N}} \lambda_{s}}+ \frac12 D_{u,z}^{\overline{N}} \lambda_{s}^2(s)\Big) ds.
\end{align*}	
Finally, since $\E\Big( e^{2pD_{u,z}^{\overline{N}} \lambda_{s}}  D_{u,z}^{\overline{N}} \lambda_{s}\Big) \leq \E\Big( \frac12 e^{4pD_{u,z}^{\overline{N}} \lambda_{s}} + \frac12 D_{u,z}^{\overline{N}} \lambda_{s}^2(s)\Big) $, from the equation \eqref{suppexpD}, and then applying the Gronwall inequality, the proof is completed.
\section{}\label{appendixF}
{\bf Malliavin calculus on Wiener-Poisson space:}\\
Let us review some concepts of Malliavin calculus on Wiener-Poisson space, as demonstrated in \cite{n19} and \cite{nunno}.  
Consider the Wiener-Poisson space $(\Omega, \mathcal{F}, P)$, and a Poisson random measure $N$ associated with a Lévy process $L$ with the Lévy measure $v$ on a complete separable metric space $(\mathbb{R}_0, \mathcal{B})$. For a positive real number $T$, let $L^{2}([0,T] \times \mathbb{R}_0^{n})$ be the space of symmetric square integrable functions on the $([0,T] \times \mathbb{R}_0^{n}, m \otimes v^{\otimes n})$, where $m$ is an atomless measure on $[0,T]$.
Denote by $\mathcal{S}$ the set of all functionals $F=\varphi(\theta_1,\theta_2,...,\theta_n)$ where $\varphi$ is a smooth function with bounded derivatives of any order and $\theta_i=\int_{0}^{T}f_i(t)dB_t$ with $f_i\in{L^{2}(\left[0,T\right])}$. Define the Wiener Malliavin derivative of $F$ by
\begin{equation*}
D^W_t{F}(w)=\sum_{i=1}^{n}\frac{\partial{\varphi}}{\partial x_i}(\theta_1,....,\theta_n)f_i(t).
\end{equation*}
For every integer $p\geq2$, the domain of the operator $D^W$, denoted by $\mathbb{D}_W^{1,p}$, is the closure of $\mathcal{S}$ with respect to the norm defined by 
\begin{equation*}\label{equ20}
\lVert{F}\rVert_{n,p}=\lVert{F}\rVert_{L^{p}(\Omega)}+\lVert{D^W {F}}\rVert_{L^{p}(\left[0,T\right]^n\times\Omega)}.
\end{equation*}
The Skorohod operator $\delta$ is the adjoint operator of $D^W$ from $L^{2}([0,T] \times \Omega)$ to $\mathbb{D}_W^{1,2}$. The following  duality relation satisfies between this two operator, which states that for given $F\in\mathbb{D}^{1,2}$ and $u \in Dom(\delta^W)$ 
\begin{equation*}
\mathbb{E}\Big(\left\langle{D^WF,u}\right \rangle_{L^{2}[0,T]}\Big):=\mathbb{E}\Big(\int_{0}^{T}(D^W_tF)u_tdt\Big)=\mathbb{E}\Big(F\delta^W(u)\Big).
\end{equation*}
For every adapted process $u$, $\delta^W(u)$ can be represented by the stochastic integral $\int_0^T u(s)dW_s$.\\
On the other hand, given $h\in L^{2}([0,T] \times \mathbb{R}_0^{n})$ and fixed $z\in \mathbb{R}_0^{n}$, we write $h(t,.,z)$ to indicate the function on $\mathbb{R}_0^{n-1}$ given by $(z_1,...,z_{n-1})\to h(t, z_1,...,z_{n-1}, z)$.
Denote by $\mathbb{D}_N^{1,2}$ the set of random variables $F$ in $L^{2}(\Omega)$ with a chaotic decomposition $F=\sum_{n=0}^{\infty}I_n(h_n)$ with $h_n\in L_s^{2}([0,T] \times \mathbb{R}_0^{n})$, satisfying
\begin{equation*}
\nonumber \sum_{n\ge1}nn!\lVert{h_n}\rVert_{L^{2}([0,T] \times \mathbb{R}_0^{n})}^{2}<\infty.
\end{equation*}
For every $F\in \mathbb{D}_N^{1,2}$, the Malliavin derivative $D^{N}F$ defines as the $L^{2}([0,T] \times \mathbb{R}_0)$-valued random variable given by
\begin{equation*}
\nonumber D_{t,z}^{N}F=\sum_{n\ge1}nI_{n-1}(h_n(t, .,z)),\;\;z\in \mathbb{R}_0.
\end{equation*}
The operator $D^{N}$ is a closed operator from $\mathbb{D}_N^{1,2}\subset L^{2}(\Omega)$ into $L^{2}(\Omega\times [0,T] \times \mathbb{R}_0)$ and satisfy the following rule.
\begin{Proposition}\cite{n19}\label{lipF}
	Let $F$ be a random variable in $\mathbb{D}^{1,2}_N$ and let $\varphi$ be a real continuous function such that $\varphi(F)$ belongs to $L^{2}(\Omega)$ and $\varphi(F + D^{N}F)$ belongs to $L^{2}(\Omega\times Z)$. Then $\varphi(F)$ belongs to $\mathbb{D}^{1,2}_N$ and
	\begin{equation}
	\nonumber D_{t,z}^{N}\varphi(F)=\varphi(F+D_{t,z}^{N}F)-\varphi(F).
	\end{equation}
\end{Proposition}
The following result characterizes $\delta^{N}$ as the adjoint operator of $D^{N}$.
\begin{Proposition}\cite{n19}\label{ddelta}
If $u\in Dom\delta^{N}$, then $\delta^{N}(u)$ is the unique element of $L^{2}(\Omega)$ such that, for all $F\in\mathbb{D}_N^{1,2}$,
\begin{equation*}
	\nonumber \mathbb{E}(\left \langle {D^{N}F,u} \right \rangle_{L^{2}([0,T] \times\mathbb{R}_0)})=\mathbb{E}(F\delta^{N}(u)).
\end{equation*}
	Conversely, if $u$ is a stochastic process in $L^{2}(\Omega\times [0,T] \times \mathbb{R}_0)$ such that, for some $G\in L^{2}(\Omega)$ and for all $F\in\mathbb{D}_N^{1,2}$,
\begin{equation*}
	\nonumber \mathbb{E}(\left \langle {D^{N}F,u} \right \rangle_{L^{2}([0,T] \times\mathbb{R}_0)})=\mathbb{E}(FG),
\end{equation*}
	then $u$ belongs to $Dom\delta^N$ and $\delta^{N}(u)=G$.
\end{Proposition}
\end{document}